\documentclass[11pt]{amsart}
\usepackage{enumerate}
\usepackage{amsmath}
\usepackage{amsthm}
\usepackage{amsfonts}
\usepackage{amssymb}
\usepackage[numbers]{natbib}
\usepackage{graphicx}
\def\A{\mathcal{A}}

\def\P{\mathcal{P}}

\usepackage{fullpage}

\def\COMMENT#1{}
\let\COMMENT=\footnote

\newtheorem{question}{Question}
\newtheorem{corollary}[question]{Corollary}

\newtheorem{conjecture}[question]{Conjecture}
\newtheorem{theorem}[question]{Theorem}

\newtheorem{proposition}[question]{Proposition}
\newtheorem{lemma}[question]{Lemma}
\newtheorem{remark}[question]{Remark}
\newtheorem{claim}[question]{Claim}
\newtheorem{definition}[question]{Definition}

\numberwithin{question}{section}
\numberwithin{equation}{section}

\begin{document}
\title{
	Existence thresholds and Ramsey properties of random posets}
\author{Victor Falgas-Ravry, Klas Markstr\"om,
 Andrew Treglown and Yi Zhao}

\thanks{VFR:  Ume{\aa} Universitet, Sweden, {\tt victor.falgas-ravry@umu.se}. Research supported by VR grant 2016-03488.\\
\indent KM: Ume{\aa} Universitet, Sweden, {\tt klas.markstrom@umu.se}. Research supported by VR grant 2014-48978.\\
\indent AT: University of Birmingham, United Kingdom, {\tt a.c.treglown@bham.ac.uk}. Research supported by EPSRC \indent grant EP/M016641/1. \\
\indent YZ: Georgia State University, Atlanta GA, USA, {\tt yzhao6@gsu.edu}. Researched supported by NSF grants \\ \indent DMS-1400073 and  DMS-1700622.}

\begin{abstract}
Let $\P(n)$ denote the power set of $[n]$, ordered by inclusion, and let $\P (n,p)$ denote the random poset obtained from $\P(n)$ by retaining each element from $\P (n)$ independently at random with probability $p$ and discarding it otherwise.

Given \emph{any} fixed poset $F$ we determine the threshold for the property that $\P(n,p)$ contains $F$ as an induced subposet. We also asymptotically determine the number of copies of a fixed poset $F$ in $\P(n)$. Finally, we obtain a number of results on the Ramsey properties of  the random poset $\P(n,p)$.

\end{abstract}
\maketitle


\section{Introduction}
Let $(P, \leq_P), (Q, \leq_Q)$ be posets. A \emph{poset homomorphism} from $(P, \leq_P)$ to $(Q, \leq_Q)$ is a function $\phi: \ P \rightarrow Q$ such that for every $x,y\in P$, if $x\leq_P y$ then $\phi(x)\leq_Q \phi(y)$. 
  We say that $(P, \leq_P)$ is a \emph{subposet} of $(Q, \leq_Q)$ if there is an injective poset homomorphism from $(P, \leq_P)$ to $(Q, \leq _Q)$; otherwise, $(Q, \leq_Q)$ is said to be \emph{$(P, \leq_P)$-free}. Further we say $(P, \leq_P)$ is an \emph{induced subposet} of $(Q, \leq_Q)$ if there is an injective poset homomorphism $\phi$ from $(P, \leq_P)$ to $(Q, \leq _Q)$ such that for every $x,y$ in $P$, $\phi(x)\leq_Q \phi(y)$ if and only if $x\leq_P y$.
We shall sometimes use $P$ as a shorthand for the poset $(P, \leq_P)$ when the partial order $\leq_P$ is clear from context, and write e.g. $P$-free for $(P, \leq_P)$-free.

Set $[n]:=\{1,2,\ldots, n\}$, and denote by $\mathcal P(n)$ the power set of $[n]$. When viewed as a poset equipped with the inclusion relation, we refer to $\mathcal P(n)$ as the \emph{Boolean lattice of dimension $n$}.
Recall that given a poset $P$,
a subset $\mathcal A \subseteq P$ is an \emph{antichain} if all distinct $A, B \in \A$ are incomparable. A \emph{chain of length $\ell$} in $P$ is an $\ell$-subset of $P$ in which all elements are comparable.

Many classical questions in graph theory have analogues in the setting of the Boolean lattice. For example, in graph theory, Tur\'an-type questions ask what is the maximum number of edges a graph on $n$ vertices may have if it does not contain any copy of a fixed graph $H$ as a subgraph. The oldest result of this flavour in the study of the Boolean lattice is Sperner's theorem~\cite{Sperner28} which asserts that the size of the largest antichain in $\mathcal P(n)$ (i.e. the maximum size of a subposet of $\P(n)$ not containing a chain of length $2$) is $\binom{n}{\lfloor n/2\rfloor}$. More generally, there has been much interest in determining the largest $P$-free subset of $\mathcal P(n)$ for a range of posets $P$ (for a  sample of such results see e.g.~\cite{AxenovichManskeMartin12, Bukh09, debon, grosz, lu, Patkos};
furthermore,~\cite{posetsurvey, katsurvey} are surveys on the topic).

 Recall that the Erd\H{o}s--R\'enyi random graph $G_{n,p}$ is the $n$-vertex graph where each edge is present with probability $p$, independently of all other choices.
In this paper we consider an analogue of the Erd\H{o}s--R\'enyi random graph in the setting of posets:
let $\P (n,p)$ be the induced subposet of $\P(n)$ obtained by independently including each element from $\P (n)$ in $\P(n,p)$ independently at random with probability $p$ (and discarding it otherwise). The random poset model $\P(n,p)$ was first investigated by R\'enyi~\cite{renyi} who determined the probability threshold for the property that $\P(n,p)$ is not itself an antichain, thereby answering a question of Erd\H{o}s. This model has also been studied with respect to a range of other properties. Answering a question of Osthus~\cite{Osthus00},   a  version of Sperner's theorem for $\P(n, p)$ was obtained independently by Balogh, Mycroft and Treglown~\cite{sper} and by Collares Neto and  Morris~\cite{cm}.\footnote{This question had first been studied by  Kohayakawa and Kreuter~\cite{kk}.} There have also been a number of results concerning the length of (the longest) chains in $\P(n,p)$ and related models of random posets (see
for example, \cite{bol,  koh, Kreuter98}).

Note too that natural questions concerning $\mathcal P(n,p)$ also arise when viewing it as a set system rather than a poset (see e.g. the random version of Katona's intersection theorem in~\cite{btw}).

\subsection{The existence threshold for a subposet}
One of the fundamental questions in the study of  the random graph $G_{n,p}$ concerns the values of $p$ for which $G_{n,p}$ with high probability (w.h.p.) contains a given fixed graph $H$ as a subgraph.
Indeed, this was the first problem studied in a seminal paper of Erd\H{o}s and R\'enyi~\cite{er1}, who determined the threshold for this problem in the case when $H$ belongs to the class of \emph{balanced} graphs. It took another twenty-one
years before Bollob\'as~\cite{boll} determined the threshold for general graphs.

It is natural to ask the analogous question in the setting of posets. That is: given a fixed poset $P$, for which values of $p$ do we have that $\P(n,p)$ w.h.p. contains a copy of $P$ as a subposet?
In this paper we answer this question for every poset $P$. More precisely, we determine the critical value $c_{\star}(P)$ such that for $p=e^{-cn}$ and $c>c_{\star}(P)$ fixed, w.h.p. $\P(n,p)$ does not contain a copy of $P$,  while for $p=e^{-cn}$ and $c<c_{\star}(P)$ fixed, w.h.p. $\P(n,p)$ contains an \emph{induced} copy of $P$.

Whilst the analogous result in the setting of the random graph $G_{n,p}$ is not too difficult to state, in the Boolean lattice we must introduce several non-trivial concepts before we can give a formal statement of our main result. Thus, we defer its precise statement
(Theorem~\ref{theorem: existence threshold}) to Section~\ref{sec:exist}. In Section~\ref{intu} we give the intuition behind this result.
We additionally prove that for almost all posets $P$ with $N$ elements, $c_{\star}(P)=(\log 2)/3+O(1/\log N)$ (see Theorem~\ref{universality}).

We remark that Kreuter~\cite{Kreuter98} considered a closely related question.
Indeed, given a \emph{distributive lattice} $L$ he determined the threshold for the property that w.h.p.  $L$ can be  \emph{embedded} in $\mathcal P(n,p)$. That is, 
given $a,b \in L$, write $a \vee b$ for the \emph{join} of $a$ and $b$ and $a \wedge b$ for the \emph{meet} of $a$ and $b$.
Then an \emph{embedding} of $L$ into $\mathcal P(n,p)$ is an injective poset homomorphism $\phi$ from $L$ to $\mathcal P(n,p)$
such that for all $a,b \in L$, $\phi (a \vee b)= \phi (a) \vee \phi (b)$ and $\phi (a \wedge b )=\phi (a) \wedge \phi (b)$.
Whilst the existence threshold we obtain for our problem has some features similar to that of Kreuter's, the two problems differ quite significantly.

\subsection{Counting subposets in the Boolean lattice}
In order to prove our existence threshold result (Theorem~\ref{theorem: existence threshold}), in Section~\ref{sec:bi} we provide a correspondence between copies of a fixed poset $P$ in $\mathcal P(n)$ and partitions of $[n]$.
As a consequence of this we  asymptotically determine the number of copies of $P$ in $\mathcal P(n)$. 
\begin{theorem}\label{thm:count}
Let $P$ be a fixed poset with $m$ antichains (including the empty antichain). Then $\mathcal{P}(n)$ contains 
$$(1+o(1))m^n$$ copies of $P$.	
\end{theorem}
Thus, 
 the number $m$ of antichains in $P$ is the parameter governing how many copies of $P$ there are in the Boolean lattice. 
Theorem~\ref{thm:count} also implies that the number of non-induced copies of $P$ in $\mathcal P(n)$ is $o(m^n)$, since such a copy of $P$ is simply another poset with strictly fewer antichains.
Note that Axenovich and Walzer~\cite[Theorem~5]{AxenovichWalzer17}  gave (asymptotically weaker) bounds
on the number of copies of $\mathcal P(t)$ in $\mathcal P(n)$.

\subsection{Ramsey properties of random posets}\label{subsection: intro ramsey}
Ramsey-type problems have been extensively studied for graphs, and there has been interest in investigating similar phenomena in the Boolean lattice. As a consequence of the Hales--Jewett theorem, we have the following analogue of Ramsey's theorem for complete graphs in the Boolean setting: given any fixed $r,m\in \mathbb N$, if $n$ is sufficiently large then in every $r$-colouring of the elements of $\mathcal P(n)$ there is a monochromatic  copy of $\mathcal P(m)$ (see e.g.~\cite[p49]{hjpbook}).
 In a recent paper, Axenovich and Walzer~\cite{AxenovichWalzer17} gave bounds on the so-called \emph{poset Ramsey number} for various posets: given posets $F$ and $F'$, the \emph{poset Ramsey number} $R(F,F')$ is the smallest $N$ such that any $2$-colouring of the elements of $\mathcal P(N)$ contains a red induced copy of $F$ or a blue induced copy of $F'$. They also considered multicolour variants of this Ramsey number. 
See also the very recent papers~\cite{lut, newboolean}.

In the case of the  random graph $G_{n,p}$, we have a clear understanding of (symmetric) Ramsey properties. Indeed, seminal work of R\"odl and Ruci\'nski~\cite{random1, random2, random3} determines the threshold for the property of $G_{n,p}$ being $(H,r)$-Ramsey for any fixed graph $H$ and  $r\in \mathbb N$. (We say that a graph $G$ is \emph{$(H,r)$-Ramsey} if every $r$-colouring of $G$ yields a monochromatic copy of $H$ in $G$.)
However, far less is known about Ramsey properties of $\P (n,p)$. 

In 1998, Kreuter~\cite{Kreuter98}  initiated the study of such questions for $\P (n,p)$ -- however, in the setting of  monochromatic embedded copies of a distributive lattice $L$.
Given a lattice $L$ and $r \in \mathbb N$ we say that a poset $P$ is \emph{$(L,r)$-embed-Ramsey} if whenever the elements of $P$ are $r$-coloured, there exists an embedded monochromatic copy of $L$ in $P$.
Kreuter~\cite{Kreuter98} determined the threshold for the property that  $\P (n,p)$ is $(C_{\ell},r)$-embed-Ramsey, where here $C_{\ell}$ denotes a chain of length $\ell$. 
He also raised the question of  generalising this result to other sublattices. In particular, he asked for the probability threshold
for the property that whenever $\P (n,p)$ is $2$-coloured it contains an embedded monochromatic copy of $\mathcal P(2)$.

Given a poset $F$ and $r \in \mathbb N$ we say that a poset $P$ is \emph{$(F,r)$-Ramsey} if whenever the elements of $P$ are $r$-coloured, there exists a monochromatic copy of $F$ in $P$.
Similarly, given posets $F_1,\dots, F_r$ we say that a poset $P$ is \emph{$(F_1,\dots, F_r)$-Ramsey} if whenever the elements of $P$ are $r$-coloured, for at least one $1\leq i \leq r$, in $P$ there exists a copy of $F_i$ in colour $i$.
In this paper we initiate the study of the following general question:
\begin{question}\label{mainques}
Given posets $F_1,\dots, F_r$, what values of $p$ ensure that w.h.p. $\P(n,p)$ is $(F_1,\dots, F_r)$-Ramsey?
\end{question}
Note that  a copy of $C_{\ell}$ in $\P(n,p)$ is always an embedded copy; so in fact the aforementioned result of Kreuter answers Question~\ref{mainques} in the case when $F_1=\dots=F_r=C_{\ell}$. As an application of our existence threshold theorem (Theorem~\ref{theorem: existence threshold}), we obtain a number of  somewhat modest results concerning the Ramsey properties of random posets. For each poset $P$ on at most $3$ elements, (combined with Kreuter's result)
we determine 
the critical value $c_{\mathrm{Ram}}(P)$ such that for $p=e^{-cn}$ and $c>c_{\mathrm{Ram}}(P)$ fixed, w.h.p. $\P(n,p)$ is not $(P,2)$-Ramsey,  while for $p=e^{-cn}$ and $c<c_{\mathrm{Ram}}(P)$ fixed, w.h.p. $\P(n,p)$ is  $(P,2)$-Ramsey. 
We give a number of results for other posets too, for example, the following result for $\mathcal P(2)$.


\begin{theorem}\label{thm1short}
The following holds:
\begin{itemize} 
\item[(i)] If $c<0.3250121326$ and $p= e^{-cn}$ then w.h.p. $\P (n,p)$ is $(\mathcal P(2),2)$-Ramsey. 
\item[(ii)] If $c>0.3289037391$ and $p= e^{-cn}$ then w.h.p. $\P (n,p)$ is not $(\mathcal P(2),2)$-Ramsey.
\end{itemize}
\end{theorem}

Further Ramsey-type results are presented in Section~\ref{sec:ramsey}.
We suspect that  Question~\ref{mainques} is likely to be extremely challenging in general. It would be very interesting to resolve the question fully in the case when $r=2$ and $F_1=F_2=\mathcal P(2)$.

\subsection{Organisation of the paper}
The paper is organised as follows. In Section~\ref{intu} we provide an intuitive outline of the existence threshold result  (Theorem~\ref{theorem: existence threshold}).
Sections~\ref{s3}--\ref{s5} introduce a number of concepts and auxiliary results that allow us to formally state and prove Theorem~\ref{theorem: existence threshold} in Section~\ref{sec:exist}.
Specifically, Section~\ref{s3} introduces the crucial notions of extension families and shadows. In Section~\ref{sec:bi} we formally introduce a correspondence between  partitions of  
$[n]$ and copies of a fixed poset $P$ in $\P(n)$. This correspondence not only allows us to prove Theorem~\ref{thm:count}, but is also vital for the proof of Theorem~\ref{theorem: existence threshold}.
In Section~\ref{s5} we introduce the notion of the weight of a partition of $[n]$ and describe its connection to shadows.

As discussed earlier, Theorem~\ref{theorem: existence threshold} determines the critical value $c_{\star}(P)$ such that for $p=e^{-cn}$ and $c>c_{\star}(P)$ fixed, w.h.p. $\P(n,p)$ does not contain a copy of $P$,  while for $p=e^{-cn}$ and $c<c_{\star}(P)$ fixed, w.h.p. $\P(n,p)$ contains an {induced} copy of $P$. Whilst we provide an explicit formula for $c_{\star}(P)$ (see Definition~\ref{def: critical density of a poset} and Remark~\ref{remarkable!}), computing $c_{\star}(P)$ by hand is in general awkward.
In Section~\ref{sec:computing} we give a number of results (and heuristics) that provide bounds for $c_{\star}(P)$; this allows us to then compute $c_{\star}(P)$ for a number of posets $P$ in Section~\ref{sec:egs}.
In Section~\ref{sec:ramsey} we turn our attention to Ramsey questions; we provide a range of results including  general bounds,  and bounds for the $(P,Q)$-Ramsey problem in $\mathcal P(n,p)$ for several specific pairs of posets $P,Q$.
We conclude the paper with a number of open problems (see Section~\ref{sec:open}).

\section{Intuition behind the existence threshold}\label{intu}
In order to state the threshold for the property that $\mathcal P(n,p)$ contains a fixed $P$ as a subposet, one requires several concepts. In particular, at first sight, it may seem rather hard to understand  the intuition behind the formal statement of the threshold. In this section we describe the threshold in more informal terms in order to build up understanding for when we do finally state the precise result.

\subsection{Notation}
All posets considered in this paper are finite. Given a poset $P$, we denote by $\mathcal{A}(P)$ the family of all \emph{antichains} of $P$. 
Note that we include the empty antichain in $\mathcal{A}(P)$. So for example, both $\emptyset$ and $\{ \emptyset \}$ belong to $\mathcal{A}(\mathcal P(n))$!
We write $a(P)$ for the cardinality of $\mathcal{A}(P)$.  We will often enumerate the set of antichains as $\mathcal{A}(P)=: \{S_1, \dots, S_{a(P)}\}$;  we always implicitly assume that $S_{a(P)}=\emptyset$.

A poset on a set of elements $X$   is said to be \emph{connected}  if it is not possible to partition $X$ into two non-empty sets $X_1$ and $X_2$  such that for every $x_1\in X_1$  and $x_2\in X_2$, the elements $x_1$ and $x_2$ are mutually incomparable.  Equivalently, a poset is connected if its Hasse diagram is a connected graph.

Given integers $m,n$, we write $[m]^{[n]}$ for the collection of ordered $m$-partitions of $[n]$, i.e. the collection of all $m$-tuples $(A_1, A_2, \ldots , A_m)$, where the $A_i$s are pairwise disjoint subsets of $[n]$ whose union is $[n]$. Further,  we write $[m]^{[n]}_{\star}$ for the collection of all $(A_1,A_2, \ldots , A_m)$ in $[m]^{[n]}$ for which $A_j \neq \emptyset$ for all $j \in [m]$.

We use standard Landau notation throughout this paper. In a probabilistic setting, we say that an $n$-dependent event $E=E(n)$ occurs \emph{with high probability (w.h.p.)} if $\mathbb{P}(E)\rightarrow 1$ as $n\rightarrow \infty$.

\subsection{A correspondence between partitions of $[n]$ and subposets of $\P (n)$}\label{subsection: intro intuition}
A crucial feature of the existence threshold concerns a correspondence between partitions of  
$[n]$ and subposets of $\P(n)$. More precisely, fix a poset $P$, and consider an arbitrary fixed ordering $S_1,\dots, S_m$ of the elements of 
$\mathcal A(P)$ (so here $m:=a(P)$). 
Now suppose we are given a partition $\mathbf{A}=(A_1, A_2, \ldots, A_m)\in [m]^{[n]}_{\star}$.  We use $\mathbf{A}$ to build an injective poset homomorphism $\phi: \ P\rightarrow \mathcal{P}_n$ as follows.  Given $i \in P$, let $D(i):=\{i' \in P: \ i'\leq_P i\}$ denote the collection of elements of $P$ which are less than or equal to $i$ with respect to the partial order $\leq_P$.
Then let $\phi$ be the map sending $i\in P$ to the subset $X_i \subseteq [n]$, where
\begin{equation}
\label{eq:Xi}
X_i :=\bigcup_{j: \ S_j\cap D(i) \ne \emptyset} A_j.
\end{equation}
Set $f(\mathbf{A})=\phi$. Roughly speaking,
in Section~\ref{sec:bi} we show that $f$ is a bijection from the set of partitions
 $[m]^{[n]}_{\star}$
 to the set of all induced copies of $P$ in $\P(n)$ (this is a slight simplification, see Lemmas~\ref{lemma: injection from hom to partitions} and~\ref{lemma: injection from partitions* to inj}).

Indeed, notice that for any $i$, $X_i=\bigcup_{i'\in D(i)} X_{i'}$ and thus $\phi$ ensures that comparable elements of $P$ are mapped
to comparable elements of $\P(n)$. Meanwhile, if $i,i'$ are incomparable elements
of $P$, since the definition of $X_i$ involves the antichain $S_j:=\{i\}$, and the definition of
$X_{i'}$ involves the antichain $S_{j'}:=\{i'\}$, we see that $i$ and
$i'$ are mapped to incomparable elements of $\P(n)$.

Given an induced copy $P'$ of $P$ in $\P(n)$ we say that ${\mathbf A}:= f^{-1} (P')\in
 [m]^{[n]}_{\star}$ is the partition of $[n]$ \emph{associated with $P'$}.
Write ${\mathbf A}=(A_1,\dots,A_m)$ and define $a_i:=|A_i|/n$.
We say that $P'$ is a  copy of $P$ of \emph{$(a_1,\dots, a_m)$-type}.

The partition associated with a copy $P'$ of $P$ in $\P(n)$ encodes structural information on how $P'$ is positioned within $\P(n)$. To illustrate this, consider the case when $P:=\P(2)$.
Let $x_1,\dots, x_4$ denote the elements of $\P(2)$, where $x_1$ is the minimal element and $x_4$ is the maximal element.
Note we have that $|\mathcal A(\P(2))|=6$. By definition, every copy $P'$ of $\P(2)$ in $\P(n)$ of $(1/6,1/6,1/6,1/6,1/6,1/6)$-type has its minimal element $x_1$ on the $n/6$-layer of $\P(n)$ (since $x_1$ lies in one antichain in $\P(2)$).
The two middle elements $x_2, x_3$ of $\P(2)$ lie in the $n/2$-layer of $\P(n)$ (since they each lie in two antichains, in addition to the antichain that $x_1$ lies in). Finally,
$x_4$ is positioned on the $5n/6$-layer (since it is only the empty antichain that does not contain one of $x_1,\dots,x_4$).
The type of $P'$ not only determines the layers of $\P(n)$ in which the elements of $P'$ are located; it also gives us information about  the way in which such elements `overlap', when viewed as subsets of $[n]$.
Indeed, continuing our running example above, suppose $x_2\in \P(2)$ is mapped to the set $X_2 \in \mathcal \P(n)$ and $x_3\in \P(2)$ is mapped to the set $X_3 \in \mathcal \P(n)$. Since $\{x_1\}$ and $\{x_2, x_3\}$ are the only antichains that intersect both $D(x_2)$ and $D(x_3)$,
we know that $\vert X_2 \cap X_3 \vert=n/3$.

Note that in such a copy $P'$ of $\P(2)$ in $\P(n)$ we have a copy $V'$ of the poset $V=\{x_1,x_2,x_3\}$ 
whose minimal element lies in the $n/6$-layer, and whose two other elements lie in the $n/2$-layer of $\P(n)$.
We have that $\mathcal A(V)= \left \{  \{ x_1 \}, \{x_2\}, \{x_3 \} , \{x_1,x_2 \} , \emptyset\right \}$.
It is easy to check that $V'$ is in fact a copy of $V$ in $\P(n)$ of  $( 1/6, 1/6,1/6   ,1/6,1/3)$-type.

\subsection{Copies of $P$ in $\mathcal P(n,p)$}
The correspondence mentioned in the previous section plays a crucial role in the existence threshold problem.
To see this, suppose we first ask for the threshold for the property that $\mathcal P(n,p)$ contains a copy of $\P(2)$ of $(1/6,1/6,1/6,1/6,1/6,1/6)$-type.
To ensure w.h.p. that $\mathcal P(n,p)$ contains such a copy of $\P(2)$, certainly we need $p$ to be chosen so that the expected number of copies of $\P(2)$ of $(1/6,1/6,1/6,1/6,1/6,1/6)$-type in $\mathcal P(n,p)$ is at least $1$.
(Otherwise, Markov's inequality easily implies one cannot have such a copy of $\P(2)$ w.h.p.)
Moreover, one also needs $p$ to be such that the expected number of copies of $V$ of the `correct type' (i.e.
$(1/6,1/6,1/6,1/6,1/3)$-type) in $\mathcal P(n,p)$ is at least 1.
In fact, for any subposet $F$ of $\P(2)$ one needs that the  expected number of copies of $F$ of the `correct' type in $\mathcal P(n,p)$ is at least 1.
It turns out these are the only barriers; if $p$ is such that each of the above expectations is large, then indeed w.h.p. $\mathcal P(n,p)$ contains a copy of $\P(2)$ of $(1/6,1/6,1/6,1/6,1/6,1/6)$-type.

Now suppose we wish to find the threshold for the property that $\mathcal P(n,p)$ contains a copy of $\P(2)$ (i.e. the type of $\P(2)$ does not matter).
Then, roughly speaking, we prove that this threshold $p^*$ is the smallest $0<p<1$ such that there is some $6$-tuple $(\alpha _1,\dots, \alpha _6)$ with each $\alpha _i >0$ such that
\begin{itemize}
\item $\sum \alpha _i =1$;
\item the expected number of copies of $\P(2)$ in $\mathcal P(n,p)$ of $(\alpha _1,\dots, \alpha _6)$-type is at least one;
\item given any subposet $F$ of $\P(2)$, the expected number of copies of $F$ of the `correct' type is at least one.
\end{itemize}

More generally, the threshold $p^*$ for the existence of \emph{any} fixed poset $P$ in $\mathcal P(n,p)$ is analogous.
Indeed, it is the smallest $0<p<1$ such that there is some $m$-tuple $(\alpha _1,\dots, \alpha _m)$ with each $\alpha _i >0$ such that
\begin{itemize}
\item $\sum \alpha _i =1$;
\item the expected number of copies of $P$ in $\mathcal P(n,p)$ of $(\alpha _1,\dots, \alpha _m)$-type is at least one;
\item given any subposet $F$ of $P$, the expected number of copies of $F$ of the `correct' type is at least one.
\end{itemize}
(Recall here we define $m:=a(P)=|\mathcal A(P)|$.)

\section{Extension families and shadows}\label{s3}
To rigorously build up the correspondence between partitions of $[n]$ and copies of a poset $P$ in $\mathcal P(n)$ described in the last section, we require several definitions. In this section we shall define notions of \emph{extension families} and \emph{shadows} for antichains that will play a crucial role in both the proofs and the statements of our main results. Here (and elsewhere in the paper) we write $x<_Py$ as a shorthand for the  statement ``$x\leq_P y$ and $x\neq y$''.

\begin{definition}[Extension family]\label{def: extension family}
Let $S$ be an antichain in $\mathcal{A}(P)$. The extension family of $S$ in $P$ is 
\[\mathrm{Ext}_P(S):=\{ S' \in \mathcal{A}(P): \ S\subsetneq S'  \}.\]
\end{definition}
\noindent  Thus $\mathrm{Ext}_P(S)$ is the collection of all antichains in $\mathcal{A}(P)$ strictly containing $S$ as a sub-antichain. So, for example, $\mathrm{Ext}_P(\emptyset)=\mathcal{A}(P)\setminus \{\emptyset\}$, and $\mathrm{Ext}_P(S)=\emptyset$ if and only if $S$ is a maximal antichain. When the poset $P$ is clear from context, we usually omit the subscript $P$ and write $\mathrm{Ext}(S)$ for $\mathrm{Ext}_P(S)$.
\begin{definition}[Shadow]\label{def: shadow of an antichain}
	Let $(P, \leq_P)$ be a poset, and $Q\subseteq P$. Given an antichain $S\in \mathcal{A}(P)$, its \emph{(upper) shadow} $\partial_Q S$ is the set of all $y\in Q$ such that
	\begin{enumerate}[(i)]
		\item there exists $x\in S$ with $x\leq_P y$,
		\item $y$ is $\leq_P$-minimal in $Q$ with respect to property (i): that is, for every $z\in Q$ with $z<_P y$ and every $x \in S$, we have $x\not\leq_P z$. 
	\end{enumerate}	
\end{definition}
\noindent Clearly we have $\partial_PS=S$ for every antichain $S\in \mathcal{A}(P)$. Going back to our running example $\P(2)$ from Section~\ref{subsection: intro intuition}, the shadow of the antichain $\{x_1\}$ inside the subposet $\Lambda$ of $\P(2)$ induced by $\{x_2, x_3, x_4\}$ is the antichain $\{x_2, x_3\}$; also, the shadow of the antichain $\{x_2, x_3\}$ inside the subposet induced by $\{x_3, x_4\}$ is the singleton antichain $\{x_3\}$. More generally we have: 
\begin{proposition}\label{prop:shadow}
{\ }
		For every $S\in \mathcal{A}(P)$ and $Q\subseteq P$, the shadow $\partial_QS$ is an antichain in $(Q, \leq_P)$ (and thus in $(P, \leq_P)$).	
\end{proposition}
\begin{proof}
Property (ii) in Definition~\ref{def: shadow of an antichain} implies elements of $\partial_QS$ are pairwise $\leq_P$-incomparable. 
\end{proof}

\section{Poset embeddings and partitions}	\label{sec:bi}
Let $(P, \leq_P), (Q, \leq_Q)$ be posets. Recall that a \emph{poset homomorphism} from $(P, \leq_P)$ to $(Q, \leq_Q)$ is a function $\phi: \ P \rightarrow Q$ such that for every $x,y\in P$, if $x\leq_P y$ then $\phi(x)\leq_Q \phi(y)$. 
We denote by $\mathrm{hom}_P(Q)$ the set of poset homomorphism from $(P, \leq_P)$ to $(Q, \leq _Q)$, and by $\mathrm{inj-hom}_P(Q)$ the set of injective homomorphisms.

We begin by showing that the number of copies of a subposet $P$ inside $\mathcal{P}(n)$ may be estimated by counting certain partitions of $[n]=\{1,2,\ldots , n\}$.	Given integers $m,n$, recall that we write $[m]^{[n]}$ for the collection of ordered $m$-partitions of $[n]$, i.e. the collection of all $m$-tuples $(A_1, A_2, \ldots , A_m)$, where the $A_i$s are pairwise disjoint subsets of $[n]$ whose union is $[n]$. Further, recall that we write $[m]^{[n]}_{\star}$ for the collection of all $(A_1,A_2, \ldots , A_m)$ in $[m]^{[n]}$ for which $A_j \neq \emptyset$ for all $j \in [m]$. Note that 
\begin{align}\label{eq: estimate on m-partitions}
	\vert [m]^{[n]}\vert =m^n\qquad \textrm{ and, for $m$ fixed } \qquad \vert[m]^{[n]}_{\star}\vert=\sum_{i=0}^{m-1}(m-i)^n\binom{m}{i}\left(-1\right)^{i}= m^n +O\left((m-1)^n\right).
\end{align}

\noindent Our goal in this section is to prove the following:
\begin{theorem}\label{theorem: counting copies of P}
Let $P$ be a fixed poset with $a(P)=m$. Then
\[ \Bigl\vert \mathrm{inj-hom}_P(\mathcal{P}(n))\Bigr\vert =m^n+O\left((m-1)^n\right).\] 	
\end{theorem}
We shall prove this result by constructing two maps: an injection from $\mathrm{hom}_P(\mathcal{P}(n))$ into $[m]^{[n]}$ (Lemma~\ref{lemma: injection from hom to partitions}), and an injection  from $[m]^{[n]}_{\star}$ into $\mathrm{inj-hom}_P(\mathcal{P}(n))$ (Lemma~\ref{lemma: injection from partitions* to inj}). In addition to proving Theorem~\ref{theorem: counting copies of P}, these maps will play an important role when determining the existence threshold for copies of $P$ in $\mathcal{P}(n)$, by establishing a correspondence between certain ``weighted'' copies of $P$ and certain ``weighted'' $m$-partitions of $[n]$.
(In the language of Section~\ref{intu} we mean the correspondence between copies of $P$ in $\P (n)$ of a given type and the associated partition of $[n]$.)

Let $(P, \leq_P)$ be a poset with $a(P)=m$. Assume without loss of generality that $P=[N]$, and that the antichains of $P$ are enumerated as 
\[\mathcal{A}(P)=\{S_1, S_2, \ldots, S_m\},\]
where $S_m=\emptyset$ is the empty antichain. Suppose we are given $\phi\in \mathrm{hom}_P(\mathcal{P}(n))$. We use $\phi$ to build an $m$-partition of $[n]$ as follows. 
\begin{enumerate}
	\item For $i\in [N]$, set $X_i:=\phi(i)$.
	\item For every $i\in [N]$, set $Y_i:=X_i \setminus \bigcup_{j<_Pi} X_j$.
	\item For every $j\in [m-1]$, set $Z_j:=\bigcap_{i \in S_j} Y_i$. 
	\item For every $j\in [m-1]$, set  $A_j:= Z_j \setminus \left(\bigcup_{S_k \in \mathrm{Ext}(S_j)} Z_k\right)$. 
	\item Finally, set $A_m := [n]\setminus \left(\bigcup_{j\in [m-1]}A_j\right)$.
\end{enumerate}
So $X_i$ is the subset of $[n]$ that $\phi$ maps $i$ to; $Y_i$ is the subset of $[n]$ that contains all elements of $X_i$ that do not lie in any other $X_j$ where $j<_P i$; given a non-empty antichain $S_j$, $Z_j$ is set of elements of $[n]$ that lie in every $Y_i$, for each $i$ in the antichain $S_j$.
Now define $f_1(\phi):=(A_1, A_2, \ldots, A_m)$.
\begin{lemma}\label{lemma: injection from hom to partitions}
	The map $f_1$ is an injection $\mathrm{hom}_P(\mathcal{P}(n))\rightarrow[m]^{[n]}$.
\end{lemma}
\begin{proof}We split the proof into two claims, from which the lemma is immediate.
\begin{claim}\label{claim: f1(phi) is a partition}
For every $\phi\in \mathrm{hom}_P(\mathcal{P}(n))$, $f_1(\phi)\in [m]^{[n]}$.	
\end{claim}	
\begin{proof}
The definition of $A_m$ ensures that $\bigcup_{j\in [m]}A_j=[n]$, so all we need to show is that the $A_j$ are pairwise disjoint. Consider $A_{j_1}$ and $A_{j_2}$	with $j_1,j_2$ distinct elements of $[m]$. 
Clearly $A_m$ is disjoint from every other $A_i$ so we may assume that $j_1,j_2 \leq m-1$.

If $S_{j_1}\in \mathrm{Ext}(S_{j_2})$, then by definition (step 4) $A_{j_2}$ is disjoint from $Z_{j_1}\supseteq A_{j_1}$. We are similarly done if $S_{j_2}\in \mathrm{Ext}(S_{j_1})$. Thus we may assume that $S_{j_1}\not\subseteq S_{j_2}$ and $S_{j_2}\not\subseteq S_{j_1}$. In particular, there exist $i_1 \in S_{j_1}\setminus S_{j_2}$ and $i_2\in S_{j_2}\setminus S_{j_1}$.

Suppose for a contradiction that there exists some $x\in [n]$ with $x\in A_{j_1}\cap A_{j_2}$. Then by definition (steps 3 and 4) $x\in A_{j_1}\subseteq Z_{j_1}=\bigcap_{i \in S_{j_1}}Y_i$. In particular we must have $x\in Y_{i_1}$. Similarly we have $x\in Y_{i_2} $. By definition (step 2), this implies the elements $i_1$ and $i_2$ are incomparable in $(P, \leq_P)$.

Now consider $S_{j_3}:=S_{j_1}\cup S_{j_2}$. The paragraph above established that elements in $S_{j_1}\setminus S_{j_2}$ and $S_{j_2}\setminus S_{j_1}$ are mutually incomparable. Together with the fact that $S_{j_1}$ and $S_{j_2}$ are antichains, this implies that $S_{j_3}$ is an antichain.

Thus $S_{j_3}$ is an antichain which extends both of $S_{j_1}$ and $S_{j_2}$. Further $S_{j_3}$ is distinct from both $S_{j_1}$ (since $i_2 \in S_{j_2}\setminus S_{j_1}\subseteq S_{j_3}\setminus S_{j_1}$) and $S_{j_2}$ (since $i_1\in S_{j_3}\setminus S_{j_2}$). What is more, since $x\in A_{j_1} \subseteq Z_{j_1} = \bigcap_{i \in S_{j_1}} Y_i$, we have that $x\in Y_i$ for every $i \in S_{j_1}$, and similarly $x\in Y_i$ for every $i \in S_{j_2}$. This implies that $x\in Z_{j_3}=\bigcap_{i \in S_{j_3}} Y_i$. Since $S_{j_3}\in \mathrm{Ext}(S_{j_1})$ and by definition (step 4) we have $A_{j_1}\subseteq Z_{j_1}\setminus Z_{j_3}\subseteq [n]\setminus \{x\}$, and $x\notin A_{j_1}$, which gives the desired contradiction.
\end{proof}

\begin{claim}\label{claim: f1 is an injection}
	$f_1$  is injective.
\end{claim}
\begin{proof}
Let $\phi, \phi'$ be distinct elements of $\mathrm{hom}_P(\mathcal{P}(n))$. Let $X_i$, $Y_i$, $Z_j$, $A_j$ and $X'_i$, $Y'_i$, $Z'_j$, $A'_j$ be the families of subsets of $[n]$ 	in steps 1--4 of the definition of $f_1$ applied to $\phi$ and $\phi'$ respectively.

Since $\phi\neq \phi'$, there must be a $\leq_P$-minimal element $i\in P$ such that $\phi(i)=X_i\neq X'_i=\phi'(i)$ and  for all $k<_P i$, $X_k=X'_k$.  The symmetric difference of $X_i \triangle X'_i$ is nonempty and we may assume without loss of generality that there exists $x\in [n]$ with $x\in X_i\setminus X'_i$. By our $\leq_P$-minimality assumption, we have $x\notin X_k$ for all $k<_P i$, and thus $x\in Y_i$.

Let $S_{j_1}$ be the antichain consisting of the singleton $\{i\}$. By definition (step 3) , we have $x\in Z_{j_1}= Y_i $. Now let $S_{j_2}$ be a $\subseteq$-maximal antichain from $\mathrm{Ext}(S_{j_1})\cup \{S_{j_1}\}$
with $x\in Z_{j_2}$ (i.e. $x\in Z_{j_2}$ and for every $S_{j_3}\in \mathrm{Ext}(S_{j_2})$, we have $x\notin Z_{j_3}$). By definition (step 4), we have $x\in A_{j_2}$. On the other hand, since $i \in S_{j_1}\subseteq S_{j_2}$, we have by definition (steps 4, 3) that $A'_{j_2}\subseteq Z'_{j_2}\subseteq Y'_i \subseteq X'_i$. Since $x\notin X'_i$, $x\notin A'_{j_2}$ and hence $A_{j_2}\neq A'_{j_2}$.
 The partitions $f_1(\phi)$ and $f_1(\phi')$ are thus different, as claimed.
\end{proof}			
\end{proof}

Now suppose we are given a partition $\mathbf{A}=(A_1, A_2, \ldots, A_m)\in [m]^{[n]}_{\star}$.  We use $\mathbf{A}$ to build an injective poset homomorphism $\phi: \ P\rightarrow \mathcal{P}(n)$ by letting $\phi(i)= X_i$, where $X_i$ defined in \eqref{eq:Xi}.
Set $f_2(\mathbf{A}):=\phi$.
\begin{lemma}\label{lemma: injection from partitions* to inj}
	The map $f_2$ is an injection $[m]^{[n]}_{\star}\rightarrow \mathrm{inj-hom}_P(\mathcal{P}(n))$. Moreover, for each $\mathbf{A} \in [m]_{\star} ^{[n]}$, $f_2(A)$ is an induced copy of $P$ in $\mathcal P(n)$.
\end{lemma}
\begin{proof}
 Again we split the proof of the lemma into two claims.	
\begin{claim}\label{claim: f2(phi) is a poset-homomorphism}
	For every $\mathbf{A}\in [m]^{[n]}_{\star}$, $\phi=f_2(\mathbf{A}) \in \mathrm{inj-hom}_P(\mathcal{P}(n))$. Moreover $\phi(P)$ is an induced copy of $P$ in $\mathcal P(n)$.
\end{claim}	
\begin{proof}
Let $\mathbf{A}\in [m]^{[n]}_{\star}$, and let $\phi=f_2(\mathbf{A})$.  If $i'\leq_P i$, then $D(i')\subseteq D(i)$, which by construction of $\phi$ implies $X_{i'}\subseteq X_i$. Thus $\phi: \ i \mapsto X_i$ is a poset homomorphism from $(P, \leq_P)$ to $(\mathcal{P}(n), \subseteq)$ as claimed.

It remains to show that $\phi$ is an injection and that $\phi(P)$ is an induced copy of $P$ in $\P(n)$.
It suffices to show that for $i_1, i_2\in P$, if $i_1\nleq _P i_2$, then $X_{i_1}\nsubseteq X_{i_2}$. 
Indeed, assuming this fact, if $i_1\ne  i_2$, then either $i_1\nleq _P i_2$ or $i_2\nleq _P i_1$, and in either case, we have $X_{i_1}\ne X_{i_2}$; in particular, this ensures $\phi $ is injective and that 
strictly comparable pairs are mapped to strictly comparable pairs. The fact also guarantees incomparable pairs are mapped to incomparable pairs, as desired.

To prove this fact
 assume that $i_1\nleq _P i_2$. Then $i_1\notin D(i_2)$. Let $S_j$ be the antichain $\{i_1\}$. Then $S_j\cap D(i_1)\ne \emptyset$ and $S_j\cap D(i_2)= \emptyset$. It follows that $A_j\subseteq X_{i_1}$ and $A_j\cap X_{i_2} = \emptyset$. 
Since $\mathbf{A}\in[m]^{[n]}_{\star}$, we have $A_j\neq \emptyset$. This implies that $X_{i_1}\nsubseteq X_{i_2}$, as claimed.
\end{proof}

\begin{claim}\label{claim: f2 is injective}
$f_2$ is injective.
\end{claim}	
\begin{proof}
Let $\mathbf{A}=(A_1, \ldots , A_m)$ and $\mathbf{A}'=(A'_1, \ldots, A'_m)$	be distinct partitions from $[m]^{[n]}_{\star}$, and let $\phi=f_2(\mathbf{A})$  and $\phi'=f_2(\mathbf{A}')$. We claim $\phi\neq \phi'$.

If $A_m \neq A'_m$, then $\bigcup_{i\in P}\phi(i)=[n]\setminus A_{m}\neq [n]\setminus A'_m=\bigcup_{i\in P}\phi'(i)$, and hence $\phi\neq \phi'$ as required. Assume therefore that $A_m=A'_m$. Since $\mathbf{A}, \mathbf{A}'$ are distinct partitions of $[n]$ in which every part is non-empty, there exists $x\in [n]$ and distinct elements $j_1, j_2 \in [m-1]$ such that $x\in A_{j_1}\setminus A'_{j_1}$ and $x\in A'_{j_2}\setminus A_{j_2}$. Now $S_{j_1},S_{j_2}$ are distinct non-empty antichains in $P$. Assume without loss of generality that there exists some element $i_1\in P$ with $i_1\in S_{j_1}\setminus S_{j_2}$.

First assume that $S_{j_2} \cap D(i_1) = \emptyset$. By the definition of $\phi$, $X'_{i_1}\subseteq [n]\setminus A'_{j_2}\subseteq [n]\setminus \{x\}$. On the other hand, since $S_{j_1}\cap D(i_1)\ne \emptyset$, we have 
$x\in A_{j_1}\subseteq X_{i_1}$. Thus $\phi(i_1)=X_{i_1}\neq X'_{i_1}=\phi'(i_1)$.

Second assume that there exists $i_2\in S_{j_2}$ with $i_2 <_P i_1$, then by definition of $f_2$ we have $x\in A'_{j_2}\subseteq X'_{i_2}$. On the other hand for every $i'\leq_P i_2$ we have $i'<_P i_1$ and thus $i'\notin S_{j_1}$, implying $X_{i_2}\subseteq [n]\setminus A_{j_1}$. 
In particular $x\notin X_{i_2}$,  and thus $\phi'(i_2)=X'_{i_2}\neq X_{i_2}=\phi(i_2)$.

%

It follows that $\phi$ and $\phi'$ are distinct members of $\mathrm{hom}_P(\mathcal{P}(n))$, as claimed.
\end{proof}
\end{proof}

\begin{remark}\label{theremark}
The proof of Lemma~\ref{lemma: injection from partitions* to inj} shows a little more, namely, $f_2$ remains an injection when viewed as a function $[m]^{[n]}\rightarrow \mathrm{hom}_P(\mathcal{P}(n))$ (because we only used $\mathbf{A}\in [m]^{[n]}_{\star}$ when showing that $f_2(\mathbf{A})$ is an  injective poset homomorphism). It is not hard to show $f_2$ is in fact the inverse of $f_1$: for every $\mathbf{A}\in [m]^{[n]}$, we have $f_1(f_2(\mathbf{A}))=\mathbf{A}$. 
\end{remark}
\begin{proof}[Proof of Theorem~\ref{theorem: counting copies of P}]
By Lemmas~\ref{lemma: injection from partitions* to inj} and~\ref{lemma: injection from hom to partitions}

\begin{align*}
\vert [m]^{[n]}_{\star} \vert  \leq \vert \mathrm{inj-hom}_P(\mathcal{P}(n))\vert \leq \vert \mathrm{hom}_P(\mathcal{P}(n))\vert  \leq \vert [m]^{[n]}\vert.
\end{align*}
The theorem then follows from the estimates (\ref{eq: estimate on m-partitions}).
\end{proof}

\section{Shadows and weighted partitions}\label{s5}
Building on the work in the previous section, we investigate weighted partitions and their interaction with shadows. Write $\triangle_m$ for the $m$-simplex
\[\triangle_m:= \left\{(\alpha_1, \alpha_2, \ldots , \alpha_m): \ \forall i, \ \alpha_i\geq 0  \textrm{ and } \sum_{i=1}^m \alpha_i=1\right\}.\]
We write $\triangle ^*_m$ for the set of all $\boldsymbol{\alpha} \in \triangle _m$ whose coordinates are all non-zero. We follow the convention of using $\boldsymbol{\alpha}$ to denote the vector $(\alpha_1, \alpha_2, \ldots , \alpha_m)\in \triangle_m$ (and vice versa). Similarly we use $\mathbf{A}$ to denote the partition $(A_1, A_2, \dots, A_m) \in [m]^{[n]}$ (and vice versa).

\begin{definition}[Weighted partitions]\label{def: weighted partitions}
	The \emph{weighting} of $\mathbf{A}\in[m]^{[n]}$ is 
	\[w(\mathbf{A}):= \left(\frac{\vert A_1\vert}{n}, \frac{\vert A_2\vert}{n},\ldots, \frac{\vert A_m\vert}{n} \right).\]
	The weighting $w(\mathbf{A})$ is an element of $\triangle_m$. Given $\boldsymbol{\alpha} \in \triangle_m$, we say that $\mathbf{A}$ is \emph{$\boldsymbol{\alpha}$-weighted} if $w(\mathbf{A})=\boldsymbol{\alpha}$.  We denote by $[m]^{[n]}_{\boldsymbol{\alpha}}$ the collection of all $\boldsymbol{\alpha}$-weighted $\mathbf{A}\in [m]^{[n]}$, and we say $\boldsymbol{\alpha}$ is \emph{feasible} for $n$ if this collection is non-empty (that is, if $n\boldsymbol{\alpha}\in \left(\mathbb{Z}_{\geq 0}\right)^m$).
\end{definition}
\begin{definition}[$\varepsilon$-close weightings]
	Given $\varepsilon>0$, we say that two elements $\boldsymbol{\alpha},  \boldsymbol{\beta} \in \triangle_m$ are \emph{$\varepsilon$-close} if 
	\[\|\boldsymbol{\alpha}-\boldsymbol{\beta} \|_{\infty}:= \max\Bigl\{ \vert \alpha_i -\beta_i\vert : \ 1\leq i \leq m\Bigr\} \leq \varepsilon.\]
	We also say that a partition $\mathbf{A}\in [m]^{[n]}$ is \emph{$\varepsilon$-close} to being $\boldsymbol{\alpha}$-weighted if its weighting $w(\mathbf{A})$ is  $\varepsilon$-close to $\boldsymbol{\alpha}$. We denote by  $[m]^{[n]}_{\boldsymbol{\alpha}\pm \boldsymbol{\varepsilon}}$ the collection of all such $\mathbf{A}$. We say $\boldsymbol{\alpha}\pm \boldsymbol{\varepsilon}$ is \emph{feasible} for $n$ if this collection is non-empty.
\end{definition}

\begin{definition}
	Let $(P, \leq_P)$ be a poset together with a labelling of its antichains as $\mathcal{A}(P)=\{S_1, S_2, \ldots, S_m\}$ (where $S_m$ is the empty antichain). Suppose we are given an ordered $m$-partition of $[n]$, $\mathbf{A}=(A_1, A_2, \ldots, A_m)$, with each set $A_i$ associated to 
	the antichain $S_i$. Given $Q\subseteq P$ and a labelling of the antichains in $(Q, \leq_P)$ as $\mathcal{A}(Q)=\{T_1, T_2, \ldots ,T_M\}$ (where $T_M$ is the empty antichain), the \emph{$M$-partition of $[n]$ inherited from $\mathbf{A}$} is $\mathbf{B}=(B_1, B_2, \ldots , B_M)$, where
	\[B_i :=\bigcup_{j: \ \partial_Q S_j=T_i} A_j.\]
	(Note that $\mathbf{B}$ is a partition of $[n]$ by Proposition~\ref{prop:shadow}.) We call $\mathbf{B}$ the \emph{$Q$-shadow of $\mathbf{A}$}, and denote it by $\partial_Q (\mathbf{A})$. 
	Note that $ A_m\subseteq B_M$.	
\end{definition}
Observe that if $\mathbf{A}$ is $\boldsymbol{\alpha}$-weighted, then its $Q$-shadow $\mathbf{B}=\partial_Q (\mathbf{A})$ is $\boldsymbol{\beta}$-weighted, where $\boldsymbol{\beta}$ is given by
\begin{align}\label{eq: inherited weighting}
\beta_i := \sum_{j: \  \partial_Q S_j=T_i} \alpha_j \qquad \forall i \in [M]. 
\end{align}
We call $\boldsymbol{\beta}$ the \emph{weighting induced by $\boldsymbol{\alpha}$ in $Q$}, or \emph{shadow} of $\boldsymbol{\alpha}$ in $Q$, and denote it by $\boldsymbol{\beta}=:\partial_Q(\boldsymbol{\alpha})$.

We can relate  $Q$-shadows to our injection $f_1: \ \mathrm{inj-hom}_P(\mathcal{P}(n))\rightarrow [m]^{[n]}$ as follows. Given  $\phi \in \mathrm{inj-hom}_{P}(\mathcal{P}(n))$, let $\phi_{\vert Q}$ denote the restriction of $\phi$ to $Q$ (which is an element of $\mathrm{inj-hom}_{Q}(\mathcal{P}(n))$). In a slight abuse of notation, we let $f_1(\phi_{\vert Q})$ denote the image of $\phi_{\vert Q}$ under $f_1$ defined with respect to $Q$.
\begin{proposition}\label{prop: Q-shadow and f1}
	Let $(P, \leq_P)$ be a poset, and let $Q\subseteq P$.  Then for every $\phi \in \mathrm{inj-hom}_{P}(\mathcal{P}(n))$, we have
	\[ \partial_Q (f_1(\phi)) = f_1(\phi_{\vert Q}).\]
\end{proposition}
\begin{proof}
Assume without loss of generality that $P=[N]$ and $Q=[q]\subseteq [N]$.  Further let  $\mathcal{A}(P)=\{S_1, S_2, \ldots, S_m\}$ and $\mathcal{A}(Q)=\{T_1, T_2, \ldots ,T_M\}$ be enumerations of the antichains in $(P, \leq_P)$ and $(Q, \leq_P)$ respectively with $S_m= T_M= \emptyset$ being the empty antichain.

To prove the proposition, we need to revisit the construction of $f_1$ and introduce some notation. Let $X_i$, $Y_i$, $Z_j$ and $A_j$ be as in the construction of $f_1$. Now
\begin{itemize}
	\item for every $i \in [q]$, set $\tilde{X}_i :=\phi_{\vert Q}(i)$ (note $\tilde{X}_i=X_i$),
	\item for every $i \in [q]$ set $\tilde{Y}_i:=\tilde{X}_i \setminus \{\tilde{X}_j: \ j \in Q, \ j<_P i \}$,
	\item for every $j\in [M-1]$ set $\tilde{Z}_j:= \bigcap_{i \in T_j} \tilde{Y}_i$,
	\item for every $j\in [M-1]$ set $\tilde{A}_j:= {\tilde{Z}}_j \setminus \left(\bigcup_{T_k \in \mathrm{Ext}_Q(T_j)} {\tilde{Z}}_k\right)$,
\end{itemize}
and finally set $\tilde{A}_M:=[n]\setminus \left(\bigcup_{j \in [M-1]}\tilde{A}_j\right)$. To prove the proposition  we must  show that for every $k\in [M]$,
 $\tilde{A}_k =\bigcup_{j: \ \partial_Q S_j=T_k} A_j$. To do this, we must first establish an important property of the partition $\mathbf{A}=(A_1,A_2, \ldots, A_m)$. 
\begin{claim}\label{claim: x in Aj implies x in Xi iff i above Aj}
For every $j \in [m-1]$
 and every $x\in A_j$,  we have that $x\in \phi(i)=X_i$ if and only if $i'\leq_P i$ for some $i'\in S_j$.	
\end{claim}
\begin{proof}
Suppose $x\in A_j$. By construction of $f_1$, for every $i'\in S_j$ we have $x\in X_{i'}=\phi(i')$. As $\phi$ is a poset homomorphism, this implies $x\in X_{i}=\phi(i)$ for every $i\in P$ with $i'\leq_P i$.

For the reverse implication: let $i\in P\setminus S_j$ be such that $x\in X_i$ (we are done if $i\in S_j$). Without loss of generality, we may assume $i$ is $\leq_P$-minimal with that property; that is, for every $i'\in P\setminus S_j$ with $i'<_P i$,  we have $x\notin X_{i'}$.

Since $x\in A_j \subseteq \bigcap_{i'\in S_j} Y_{i'}$, we must have $x\in Y_{i'}$ for every $i'\in S_j$. In particular $x\notin X_{i''}$ for any $i''\in P$ with $i''<_P i'$ for some $i'\in S_j$. Thus $i$ cannot be below any element of $S_j$ in the partial order $\leq_P$.

Suppose for contradiction that $i$ was incomparable with every element of $S_j$. Then $S_{j'}=S_j\cup\{i\}$ is an antichain in $P$ extending $S_j$. Further, since $i'\not\leq_P i$ for any $i' \in S_j$ and since $x\notin X_{i'}$ for any $i'\in P\setminus S_j$ with $i'<_P i$ (by our minimality assumption on $i$), we have that $x\in Y_i$. We already know $x\in Z_j$, so we deduce that $x\in Z_{j'}=Z_j\cap Y_i$. But this implies $x\notin A_j \subseteq Z_j \setminus Z_{j'}$, a contradiction.

It follows that $i'\leq_P i$ for some $i'\in S_j$, as claimed.
\end{proof}
\begin{claim}\label{claim: Aj is included in shadow(Sj)'s set}
Let $j\in [m]$.  If $\partial_Q S_j = T_k$, then $A_j \subseteq \tilde{A}_k $.	
\end{claim}
\begin{proof}
First suppose that $j=m$. 
In Remark~\ref{theremark} we observed that $f_2$ is the inverse of $f_1$. In particular, by definition of $f_2$, $A_m$ is precisely the set of $x \in [n]$ such that 
$ x \not \in \phi (i)$ for all $i \in P=[N]$. But for each $i \in Q$ we have that $\phi (i)=X_i =\tilde{X}_i$. Thus by definition of the $\tilde{A}_{j'}$, this means $x\in A_m$ cannot be an element in $\tilde{A}_{j'}$ for any $j' \in [M-1]$. That is, $A_m \subseteq \tilde{A}_M $.
Now by definition $\partial_Q S_M =\emptyset = T_M$.
So this proves the claim in this case.

We may therefore assume that $j <m$. Now suppose that $\partial_Q S_j=T_k=\emptyset$, i.e. that $k=M$.  Suppose for a contradiction that there exists $x\in [n]$ such that $x\in A_j$ and $x \not \in \tilde{A}_M$. Then $x \in \tilde{A}_{k'}$ for some $k' <M$; this further implies
 $x \in  \tilde{X}_{i}=\phi(i) $ for some $i \in [q]=Q$. We may assume $i$ is $\leq_Q$-minimal in $Q$ with respect to this property. By Claim~\ref{claim: x in Aj implies x in Xi iff i above Aj}, there is some $i' \in S_j$ such that $i' \leq _P i$.
This property together with the definition of $\partial_Q S_j$ (and our assumption of $\leq_Q$-minimality for $i$) ensures $\partial_Q S_j\not =\emptyset$, a contradiction. Therefore $x  \in \tilde{A}_M$.
Thus $A_j \subseteq  \tilde{A}_{M}$, as desired.

Finally, suppose $T_k=\partial_Q S_j$ with $T_k, S_j\neq \emptyset$. By definition of $T_k=\partial_Q S_j$, for every $i'\in T_k$ there exists $i\in S_j$ with $i\leq_P i'$. By Claim~\ref{claim: x in Aj implies x in Xi iff i above Aj}, this implies $x\in \tilde{X}_{i'}=X_{i'}$ for every $i'\in T_k$.

Now consider $i''\in T_k$ and $i' \in Q\setminus T_k$ with $i'<_P i''$. If $x\in \tilde{X}_{i'}=X_{i'}$, then we must have $i\leq_P i'$ for some $i\in S_j$, contradicting $i'' \in T_k=\partial_Q S_j$. Thus $x\notin \tilde{X}_{i'}$ for any $i'\in Q$ such that $i'<_P i''$, and hence $x\in \tilde{Y}_{i''}$.  This in turn implies $x\in \tilde{Z}_k=\bigcap_{i''\in T_k} \tilde{Y}_{i''}$.

Now for any extension $T_{k'}$ of $T_k$ in $Q$ there exists $i''\in T_{k'}\setminus T_k$. Suppose $x\in \tilde{X}_{i''}=X_{i''}$. By Claim~\ref{claim: x in Aj implies x in Xi iff i above Aj}, this implies there exists $i\in S_j$ with $i\leq_P i''$. Since $i''\in Q$ and $i''\notin T_k=\partial_Q(S_j)$, this implies there exists $i'\in T_k$ with $i' <_P i''$ by the definition of $\partial_Q(S_j)$, contradicting  the fact that $T_{k'}$ is an antichain. Thus for any extension $T_{k'}$ of $T_k$ in $Q$ and every $i'' \in T_{k'}\setminus T_k$, $x\notin \tilde{X}_{i''}$, implying in turn that $x\notin \tilde{Z}_{k'}$.

In particular we have shown that $x\in \tilde{A}_k=\tilde{Z}_k \setminus \left(\bigcup_{T_{k'}\in \mathrm{Ext}_Q(T_k)} \tilde{Z}_{k'}\right)$. Since $x\in A_j$ was arbitrary, we have $A_j \subseteq \tilde{A}_k$ as claimed.
\end{proof}	
As $f_1(\phi)=(A_1, A_2, \ldots, A_m)$ and $f_1(\phi_{\vert Q})=(\tilde{A_1}, \tilde{A}_2, \ldots, \tilde{A}_M)$ both form partitions of $[n]$ (by Lemma~\ref{lemma: injection from hom to partitions}), Claim~\ref{claim: Aj is included in shadow(Sj)'s set} immediately implies that $\tilde{A}_k = \bigcup_{j: \partial_Q S_j=T_k}A_j$ for all $k\in [M]$. It follows that $\partial_Q (f_1(\phi)) =f_1(\phi_{\vert Q})$ as desired, proving the proposition.
\end{proof}

Having made these definitions and related shadows to partitions for subposets, our final goal in this section is to estimate the number of $\boldsymbol{\alpha}$-weighted partitions in $[m]^{[n]}$. For this purpose, we introduce the \emph{entropy} of a weighting $\boldsymbol{\alpha}\in \triangle_m$ to be
\begin{align}\label{eqdef: entropy}
H_m(\boldsymbol{\alpha}):= \sum_{j=1}^m -\alpha_j \log \alpha_j.
\end{align}
(Note that here $\log$ denotes the natural logarithm and $0 \log 0:=0$.)
The entropy function $H_m$ is a well-studied object in combinatorics and discrete probability. It has a maximum value of $\log m$ in $\triangle_m$, uniquely attained at $\boldsymbol{\alpha}=\left(\frac{1}{m}, \frac{1}{m}, \ldots \frac{1}{m}\right)$.
\begin{proposition}\label{prop: estimating the number of alpha-weighted partitions}
	Let $m\in \mathbb N$ and $\boldsymbol{\alpha} \in \triangle_m$ be fixed. 

	Then for any sequence $\varepsilon=\varepsilon(n)\rightarrow 0$ such that $\boldsymbol{\alpha}\pm \boldsymbol{\varepsilon}$ is feasible for every $n$, we have
	\begin{align*}
	\Bigr\vert [m]^{[n]}_{\boldsymbol{\alpha}\pm \boldsymbol{\varepsilon}}\Bigl\vert 
	= \exp\Bigl(H_m(\boldsymbol{\alpha})n +O\left(\log n\right)\Bigr).\notag
	\end{align*}
\end{proposition}
\begin{proof}
For every $\boldsymbol{\beta} \in \triangle_m$ which is feasible for $n$, we have by Stirling's estimate for the factorial that
	\begin{align}\label{eq: estimate for number of beta-weighted partitions}
	\Bigr\vert [m]^{[n]}_{\boldsymbol{\beta}}\Bigl\vert 
	= \binom{n}{\beta_1n, \beta_2n, \ldots, \beta_{m-1}n}
	= \exp\Bigl(H_m(\boldsymbol{\beta})n +O(\log n)\Bigr).
	\end{align}
Now there are at most 
\begin{align}\label{eq: estimate for number of feasible eps-near weightings}
\left(2\varepsilon n+1\right)^{m-1}= e^{O(\log \varepsilon n)}
\end{align} weightings $\boldsymbol{b}$ which are both $\boldsymbol{\varepsilon}$-close to $\boldsymbol{\alpha}$ and feasible for $m$. Let $\mathcal{B}$ denote the family of all such $\boldsymbol{\beta}$. By continuity of the function $H_m$ and the fact that $\varepsilon=o(1)$, for all $\boldsymbol{\beta}\in \mathcal{B}$ we have $H_m(\boldsymbol{\beta})=H_m(\boldsymbol{\alpha})+o(1)$. Putting it all together, we have
	\begin{align*}
	\Bigr\vert [m]^{[n]}_{\boldsymbol{\alpha}\pm \boldsymbol{\varepsilon}}\Bigl\vert = \sum_{\boldsymbol{\beta}\in \mathcal{B}}  \Bigr\vert [m]^{[n]}_{\boldsymbol{\beta}}\Bigl\vert 
=\Bigl \vert \mathcal{B} \Bigr\vert \exp\Bigl(H_m(\boldsymbol{\alpha})n +O( \log n)\Bigr)=\exp\Bigl(H_m(\boldsymbol{\alpha})n +O(\log n)\Bigr),
	\end{align*}
	as desired.
\end{proof}

\section{The existence threshold}\label{sec:exist}

\subsection{The existence threshold}
In this section, we shall determine the existence threshold for copies of a fixed poset $(P, \leq_P)$ in a random subposet of $(\mathcal{P}(n), \subseteq)$. Explicitly, let $p=e^{-cn}$. We let $\mathcal{P}(n,p)$ be a $p$-random subset of $\mathcal{P}(n)$, obtained by retaining each element of $\mathcal{P}(n)$ independently at random with probability $p$. This gives rise to a random poset $(\mathcal{P}(n,p), \subseteq)$. For which $c$ does this poset contain w.h.p. a copy of $P$ --- i.e. for which $c$ is $\mathrm{inj-hom}_P(\mathcal{P}(n,p))$ w.h.p. non-empty?

To answer this question, we need to introduce two definitions.
\begin{definition}\label{def: density of a subposet}
Let $(P, \leq_P)$ be a poset with $a(P) =m$. Let $Q\subseteq P$ be a subposet of $P$	with $Q\neq \emptyset$ and $a(Q)=M$.  Given a weighting $\boldsymbol{\alpha}\in \triangle_m$, let $\boldsymbol{\beta}=\partial_Q (\boldsymbol{\alpha})$ be the weighting from $\triangle_M$ induced by $\boldsymbol{\alpha}$ in $Q$ (that is, $\boldsymbol{\beta}$ is the weighting obtained when applying (\ref{eq: inherited weighting})). 
We define the \emph{critical exponent of $Q$ in $P$ with respect to $\boldsymbol{\alpha}$} to be
\begin{align*}
c_{\boldsymbol{\alpha},P}(Q):= \frac{H_M(\boldsymbol{\beta})}{\vert Q\vert}.
\end{align*}
\end{definition}
\begin{definition}\label{def: critical density of a poset}
Let $(P, \leq_P)$ be a poset with $a(P)=m$. The \emph{critical exponent} of $P$ is defined to be
\begin{align*}
c_{\star}(P):= \sup \Bigl\{c\in \mathbb{R}_{\geq 0}: \ \exists \boldsymbol{\alpha}\in \triangle_m\textrm{ s.t. } \forall Q: \emptyset\neq Q\subseteq P, \  c\leq c_{\boldsymbol{\alpha},P}(Q)\Bigr\}.
\end{align*}
\end{definition}
\begin{remark}\label{remarkable!}
	Equivalently, since $\triangle_m$ is compact and $\min\{c_{\boldsymbol{\alpha}, P}(Q): \ \emptyset \neq Q\subseteq P\}$ is a continuous function of $\boldsymbol{\alpha}$, we can express the critical exponent as:
	\begin{align*}
	c_{\star}(P)= \max_{\boldsymbol{\alpha} \in \triangle_m} \min_{\emptyset \neq Q\subseteq P} c_{\boldsymbol{\alpha},P}(Q).
	\end{align*}
\end{remark}

\begin{theorem}\label{theorem: existence threshold}
Let $(P, \leq_P)$ be a finite poset. For $c>0$ fixed and $p=p(n)=e^{-cn}$, the following hold:	
\begin{enumerate}[(i)]
	\item  if $c> c_{\star}(P)$, then w.h.p. $(P, \leq_P)$ is not a subposet of $(\mathcal{P}(n,p), \subseteq)$;
	\item if $c< c_{\star}(P)$, then w.h.p. $(P, \leq_P)$ is an induced subposet of $(\mathcal{P}(n,p), \subseteq)$.
\end{enumerate}	
\end{theorem}
\noindent The proof of parts (i) and (ii) of Theorem~\ref{theorem: existence threshold} occupy the next two subsections. Before we dive into these proofs, however, we should like to outline the main idea behind Theorem~\ref{theorem: existence threshold}, which is to look at certain ``weighted'' copies of $P$ in $\mathcal{P}(n)$.

\begin{definition}[Weighted copies of posets]\label{def: weighted subposet}
	Let $(P, \leq_P)$ be a poset with $a(P)=m$. We define the \emph{weight} of $\phi \in \mathrm{inj-hom}_P(\mathcal{P}(n))$ to be
	\[w_{\phi}:=w(f_1(\phi)),\]
	where $f_1$ is the injection $\mathrm{inj-hom}_P(\mathcal{P}(n))\rightarrow [m]^{[n]}$ given in Lemma~\ref{lemma: injection from hom to partitions} and $w$ is the weighting from Definition~\ref{def: weighted partitions}. 
	Given $\boldsymbol{\alpha}\in \triangle_m$, we say $\phi$ is an \emph{$\boldsymbol{\alpha}$-weighted} copy of $P$ in $\mathcal{P}(n)$ if $w_{\phi}=\boldsymbol{\alpha}$.
\end{definition}
\noindent By \eqref{eq: estimate for number of beta-weighted partitions} the expected number of $\boldsymbol{\alpha}$-weighted copies of $(P, \leq_P)$ in $(\mathcal{P}(n,p), \subseteq)$  is 
\[e^{H_m(\alpha)n +O(\log n)} p^{\vert P \vert}= e^{\vert P\vert \left(c_{\boldsymbol{\alpha}, P}(P)-c\right) n+O(\log n)}.\]
Thus certainly for a fixed feasible weighting $\boldsymbol{\alpha}$, if $c>c_{\boldsymbol{\alpha}, P}(P)$ then w.h.p. no such $\boldsymbol{\alpha}$-weighted copies exist. Further by Proposition~\ref{prop: Q-shadow and f1}, an $\boldsymbol{\alpha}$-weighted copy of $P$ can only exist in $\mathcal P(n,p)$ 
 if for every $Q\subseteq P$ a $\boldsymbol{\beta}$-weighted copy of $Q$ exists, where $\boldsymbol{\beta}=\partial_Q(\boldsymbol{\alpha})$. This leads us to require $ c\leq \min_{Q\subseteq P} c_{\boldsymbol{\alpha}, P}(Q)$, and to the statement of the theorem.

\subsection{Proof of Theorem~\ref{theorem: existence threshold}, part (i)}
\label{sec:54i}
Let $(P, \leq_P)$ be a poset with $a(P)=m$. Suppose $c=c_{\star}(P)+\eta$, for some $\eta>0$. For every $Q\subseteq P$, both $c_{\boldsymbol{\alpha}, P}(Q)$ and $\partial_Q(\boldsymbol{\alpha})$ are continuous functions of $\boldsymbol{\alpha}$ in the compact set $\triangle_m$. Since there are finitely many subposets $Q$: $\emptyset \neq Q \subseteq P$,
there exist constants $\varepsilon_1, \varepsilon_2>0$ such that  if $\|\boldsymbol{\alpha}-\boldsymbol{\alpha}'\|_{\infty}<\varepsilon_1$, 
then for every $Q\subseteq P$
\begin{align}\label{eq: alphas are eps1-close imply betas eps2-close}
\|\partial_Q(\boldsymbol{\alpha})-\partial_Q(\boldsymbol{\alpha}')\|_{\infty}<\varepsilon_2 \end{align}
and
\begin{align}\label{eq: alphas are eps1-close imply c(alpha, P)q are eta/2-close}
\vert c_{\boldsymbol{\alpha},P}(Q)-c_{\boldsymbol{\alpha}',P}(Q)\vert <\frac{\eta}{2} \end{align}
both hold.

For $\boldsymbol{\alpha}\in \triangle_m$, denote by $B_{\varepsilon_1}(\boldsymbol{\alpha})$ the open $\ell_{\infty}$-ball in $\triangle_m$ of radius $\varepsilon_1$ centered at $\boldsymbol{\alpha}$.
As $\triangle_m$ is compact, there exists some finite set $\mathcal{C} \subseteq \triangle_m$ such that the collection  $\Bigl\{ B_{\varepsilon_1}(\boldsymbol{\alpha}): \ \boldsymbol{\alpha} \in \mathcal{C}\Bigr\}$ constitutes an open cover for $\triangle_m$. (In fact, since $\triangle_m$ has measure $1/m!$ and each $B_{\varepsilon_1}(\boldsymbol{\alpha})$ has measure at least $(\varepsilon_1/m)^{m-1}$, it is not hard to show that one can take $\vert \mathcal{C} \vert \leq \left(\frac{C}{\varepsilon_1}\right)^m$ for some absolute constant $C>0$.)

Now fix $n\in \mathbb{N}$. For each $\boldsymbol{\alpha}\in \mathcal{C}$, let $\tilde{B}_{\boldsymbol{\alpha}}$ denote the collection of $\boldsymbol{\alpha}'$ in $B_{\varepsilon_1}(\boldsymbol{\alpha})$ which are feasible for $n$. Pick $\boldsymbol{\alpha} \in \mathcal{C}$.  By definition of $c_{\star}(P)$, we have that $c_{\boldsymbol{\alpha}, P}(Q)\leq c_{\star}(P)$ for some $Q\in P$. Set $M:=a(Q)$. For every $\boldsymbol{\alpha}' \in \tilde{B}_{\boldsymbol{\alpha}}$, we have 
$\| \boldsymbol{\alpha}'-\boldsymbol{\alpha}\|_\infty <\varepsilon_1$. 
By (\ref{eq: alphas are eps1-close imply betas eps2-close}) this implies that 
$\partial_Q(\boldsymbol{\alpha}')$ is $\varepsilon_2$-close
 to $\partial_Q(\boldsymbol{\alpha})$. Further by (\ref{eq: alphas are eps1-close imply c(alpha, P)q are eta/2-close}) we have 
\begin{align}\label{eq: bound on c(alpha, P)Q in eps1-n'hood}
	c_{\boldsymbol{\alpha}', P}(Q)\leq c_{\boldsymbol{\alpha}, P}(Q)+\frac{\eta}{2}\leq c_{\star}(P)+\frac{\eta}{2}=c -\frac{\eta}{2}.
\end{align}
Combining (\ref{eq: alphas are eps1-close imply betas eps2-close}), (\ref{eq: estimate for number of feasible eps-near weightings}), (\ref{eq: estimate for number of beta-weighted partitions}), and (\ref{eq: bound on c(alpha, P)Q in eps1-n'hood}), the expected number of $\boldsymbol{\beta}$-weighted injective poset homomorphisms $\phi\in \mathrm{inj-hom}_Q(\mathcal{P}(n,p))$ with $\boldsymbol{\beta}\in \{\partial_Q(\boldsymbol{\alpha}'): \ \boldsymbol{\alpha}'\in \tilde{B}_{\boldsymbol{\alpha}}\}$ is at most
\begin{align*}
(2\varepsilon_2n +1)^{M-1}e^{\vert Q\vert \left(c-\frac{\eta}{2}\right) n +O(\log n)}p^{\vert Q\vert}= e^{- \frac{\vert Q\vert \eta }{2}n +O(\log n)}=o(1).
\end{align*} 
By Markov's inequality, we deduce that w.h.p. no such copy exists, which in turns implies by Proposition~\ref{prop: Q-shadow and f1} that w.h.p. $\mathcal{P}(n,p)$ contains no $\boldsymbol{\alpha}'$-weighted copy of $P$ for $\boldsymbol{\alpha}'\in \tilde{B}_{\boldsymbol{\alpha}}$. Since $\boldsymbol{\alpha} \in \mathcal{C}$ was arbitrary and $\mathcal{C}$ is finite, we deduce that w.h.p. $\mathcal{P}(n,p)$ contains no $\boldsymbol{\alpha}'$-weighted copy of $P$ for $\boldsymbol{\alpha}'\in \bigcup_{\boldsymbol{\alpha}\in \mathcal{C}} \tilde{B}_{\boldsymbol{\alpha}}$. By construction of $\mathcal{C}$, this latter union covers all weightings $\boldsymbol{\alpha}' \in \triangle_m$ which are feasible for $n$. Thus we deduce that for $c=c_{\star}(P)+\eta$, $\eta>0$, and $p=e^{-cn}$  the random poset $\mathcal{P}(n,p)$ is w.h.p. $P$-free. This concludes the proof of Theorem~\ref{theorem: existence threshold}, part (i). 
\qed

\subsection{Proof of Theorem~\ref{theorem: existence threshold}, part (ii)}
Let $(P, \leq_P)$ be a poset with $a(P)=m$. Suppose $c=c_{\star}(P)-\eta$ for some $\eta>0$. By definition of $c_{\star}(P)$, there exists $\boldsymbol{\alpha}_{\star}\in \triangle _m$ 
such that for all $Q$: $\emptyset \neq Q\subseteq P$, we have $c_{\boldsymbol{\alpha}_{\star}, P}(Q)\geq c_{\star}(P)$. 
When $n$ is sufficiently large there exists  a weighting $\boldsymbol{\alpha}_f=\boldsymbol{\alpha}_f(n)\in \triangle ^*_m$ (i.e. $\boldsymbol{\alpha}_f$ has non-zero coordinates)  such that $\boldsymbol{\alpha}_f$ is feasible for $n$ and 
$\|\boldsymbol{\alpha}_f - \boldsymbol{\alpha}_{\star}\|_{\infty}<\varepsilon_1$, where $\varepsilon_1>0$ is the constant given in Section~\ref{sec:54i}. 
By \eqref{eq: alphas are eps1-close imply c(alpha, P)q are eta/2-close}, for all non-empty $Q\subseteq P$, we have
\begin{align}\label{eq: lower bound on c(alpha-feasible, P)Q}
	c_{\boldsymbol{\alpha}_{f}, P}(Q)\geq c_{\star}(P)-\frac{\eta}{2}= c + \frac{\eta}{2}.
\end{align}
Write $\Phi_{\boldsymbol{\alpha}_f, Q}(\mathcal{P}(n))$ and $\Phi_{\boldsymbol{\alpha}_f, Q}(\mathcal{P}(n,p))$ for the collections of all $\partial_Q(\boldsymbol{\alpha}_f)$-weighted $\phi\in \mathrm{inj-hom}_Q(\mathcal{P}(n))$ and $\phi\in \mathrm{inj-hom}_Q(\mathcal{P}(n, p))$ respectively. By (\ref{eq: estimate for number of beta-weighted partitions}), for all non-empty $Q\subseteq P$,
\begin{align}\label{eq: expectation lower bound for weighted Q-copies}
\mathbb{E} \bigl\vert \Phi_{\boldsymbol{\alpha_f}, Q}(\mathcal{P}(n,p)) \bigr\vert &= \bigl\vert \Phi_{\boldsymbol{\alpha_f}, Q}(\mathcal{P}(n)) \bigr\vert p^{\vert Q\vert} 
=e^{\vert Q\vert c_{\boldsymbol{\alpha}_f, P}(Q)n +O(\log n)}e^{-c\vert Q\vert n}\stackrel{(\ref{eq: lower bound on c(alpha-feasible, P)Q})}{\geq} e^{\frac{\vert Q\vert \eta}{2}n +O(\log n)}, 
\end{align}
which tends to infinity as $n\rightarrow \infty$. We shall use the celebrated Janson inequalities to show $\mathbb{P}\Bigl(\bigl\vert \Phi_{\boldsymbol{\alpha_f}, P}(\mathcal{P}(n,p)) \bigr\vert =0\Bigr)$ is small. To do this, we must first introduce some notation.

Given $\phi_1, \phi_2\in \Phi_{\boldsymbol{\alpha_f}, P}(\mathcal{P}(n))$ and non-empty $Q_1, Q_2\subseteq P$, we say that the pair $(\phi_1, \phi_2)$ is \emph{$(Q_1, Q_2)$-intersecting} if $\phi_1(Q_1)=\phi_2(Q_2)=\phi_1(P)\cap\phi_2(P)$. Let $I(Q_1, Q_2)$ denote the collection of all $(Q_1,Q_2)$-intersecting pairs $(\phi_1, \phi_2)$, and define
\[D:= \sum_{Q_1, Q_2\neq \emptyset}\quad \left(\sum_{(\phi_1, \phi_2)\in I(Q_1, Q_2)} \mathbb{P}\Bigl(\phi_1(P)\cup \phi_2(P)\subseteq \mathcal{P}(n,p) \Bigr)\right).\]
Set $\mu:=\mathbb{E} \bigl\vert \Phi_{\boldsymbol{\alpha_f}, P}(\mathcal{P}(n,p)) \bigr\vert$. We now apply the following inequalities due to Janson (see e.g. \cite{JansonLuczakRucinski90}). 

\begin{proposition}[Janson inequalities]\label{prop: Janson inequalities}
\[\mathbb{P} \Bigl(\bigl\vert \Phi_{\boldsymbol{\alpha_f}, P}(\mathcal{P}(n,p)) \bigr\vert=0\Bigr) \leq \left\{\begin{array}{ll}
e^{-\frac{\mu^2}{2D}} & \textrm{if } D\geq \mu, \\
e^{-\mu + \frac{D}{2}} & \textrm{ otherwise.}
\end{array} \right.\]	
\end{proposition}

\noindent It thus remains to bound $D$. Observe first of all that
\begin{align}\label{eq: bound on I(Q1, Q2)}
I(Q_1, Q_2)= {\bigl\vert \Phi_{\boldsymbol{\alpha_f}, P}(\mathcal{P}(n)) \bigr\vert}^2 \Bigl/\bigl\vert \Phi_{\boldsymbol{\alpha_f}, Q_1}(\mathcal{P}(n)) \bigr\vert, 
\end{align}
since each copy of $P$ in $\mathcal{P}(n)$ specifies a unique copy of $Q_1$ and $Q_2$, and since all copies `look the same', being given by partitions with exactly the same weight. Next, note that for any $(Q_1, Q_2)$-intersecting pair $(\phi_1, \phi_2)$ we have
\begin{align}\label{eq: bound on Prob(phi1 and phi2 both in)}
\mathbb{P}\Bigl(\phi_1(P)\cup \phi_2(P)\subseteq \mathcal{P}(n,p) \Bigr)= p^{2\vert P\vert - \vert Q_1\vert}.
\end{align}
Putting (\ref{eq: bound on I(Q1, Q2)}), (\ref{eq: bound on Prob(phi1 and phi2 both in)}) and (\ref{eq: expectation lower bound for weighted Q-copies}) together, we have
\begin{align}\label{eq: bound on I(Q1, Q2) summands in Delta}
\sum_{(\phi_1, \phi_2)\in I(Q_1, Q_2)} \mathbb{P}\Bigl(\phi_1(P)\cup \phi_2(P)\subseteq \mathcal{P}(n,p) \Bigr)&= \frac{{\mu}^2}{\mathbb{E} \bigl\vert \Phi_{\boldsymbol{\alpha_f}, Q_1}(\mathcal{P}(n,p)) \bigr\vert}\leq \mu^2 e^{-\frac{\vert Q_1\vert\eta}{2}n+O(\log n) }.\notag
\end{align}
Plugging this back into the definition of $D$ we have 
\begin{align*}
D\leq \sum_{Q_1\subseteq P: \ Q_1\neq \emptyset}\sum_{Q_2\subseteq P: \ \vert Q_2\vert =\vert Q_1 \vert} \quad \mu^2 e^{-\frac{\vert Q_1\vert\eta}{2}n+O(\log n) }< 2^{2\vert P\vert} \mu^2 e^{-\frac{\eta}{2}n+O(\log n)}.
\end{align*}
In particular, we have $\exp\left(-\mu^2/2D\right)=o(1)$. 
When $D< \mu$, since we showed in (\ref{eq: expectation lower bound for weighted Q-copies}) that $\mu \rightarrow \infty$ (as $n \rightarrow \infty$), we also have $\exp(-\mu + D/2)\le 
\exp\left(-\mu/2\right)=o(1)$. It then follows from the Janson inequalities (Proposition~\ref{prop: Janson inequalities}) that 
\[\mathbb{P} \Bigl(\bigl\vert \Phi_{\boldsymbol{\alpha_f}, P}(\mathcal{P}(n,p)) \bigr\vert=0\Bigr)=o(1).\]
Thus for $c=c_{\star}(P)-\eta$, $\eta>0$, and $p=e^{-cn}$,  the random poset $\mathcal{P}(n,p)$ contains w.h.p. the image of an $\boldsymbol{\alpha}_f$-weighted injective poset homorphism  $\phi: \ P\rightarrow \mathcal{P}(n)$. In particular, $P$ is w.h.p. a subposet of $\mathcal{P}(n,p)$. 

Moreover, recall that $\boldsymbol{\alpha}_f \in \triangle ^*_m$. That is, the weight  of the copy $P'$ of $P$ we obtain in $\mathcal P(n,p)$ lies in $[m]^*$.
Since $f_2$ is the inverse of $f_1$, the moreover part of Lemma~\ref{lemma: injection from partitions* to inj} implies that $P'$ is in fact an induced copy of $P$.
This concludes the proof of Theorem~\ref{theorem: existence threshold}, part (ii). 
\qed

\begin{remark}\label{remark: narrow window}
The proof of Theorem~\ref{theorem: existence threshold} parts (i) and (ii) shows something slightly stronger than we claimed. Namely, instead of having $\eta>0$ fixed, we can run through the same arguments with $\eta= K\log n /n $ for some sufficiently large constant $K>0$. A little analysis shows we can take values $\varepsilon_1, \varepsilon_2=O(\eta)$ and still satisfy \eqref{eq: alphas are eps1-close imply betas eps2-close} and~\eqref{eq: alphas are eps1-close imply c(alpha, P)q are eta/2-close}. Using the bound $\vert \mathcal{C}\vert =O\left((\varepsilon_1)^{-m}\right)$  and choosing $K$ sufficiently large to beat the $O(\log n)$ error terms, one then gets that 
\begin{itemize}
	\item if $p\leq e^{-c_{\star}(P)n -K\log n}$, then w.h.p. $\P(n,p)$ contains no copy of $P$;
	\item if $p\geq e^{-c_{\star}(P)n +K\log n}$, then w.h.p. $\P(n,p)$ contains an induced copy of $P$.
\end{itemize}
\end{remark}

\begin{remark}\label{general threshold}
The proof of Theorem~\ref{theorem: existence threshold}   can also be used to derive a more general result about the existence of induced copies of $P$  with a specific embedding in $\mathcal{P}(n)$.    Given a weighting 
$\boldsymbol{\alpha},$ we may define
$c_{\boldsymbol{\alpha}}(P):=  \min_{\emptyset \neq Q\subseteq P} c_{\boldsymbol{\alpha},P}(Q)$.
Our proof   demonstrates that   $c_{\boldsymbol{\alpha}}(P)$  is the threshold for the existence of an $\boldsymbol{\alpha}$-weighted copy of $P$ in $(\mathcal{P}(n,p), \subseteq)$. 

Here we also note that since $c_{\boldsymbol{\alpha}}(P)$ is continuous in $\boldsymbol{\alpha}$, there will be many different embeddings of  $\mathcal{P}$ in $(\mathcal{P}(n,p), \subseteq)$ as soon as $p$ is strictly larger 
than the value given by $c_{\star}(P)$.

\end{remark}

\section{Computing {$c_{\star}(P)$} in practice: general bounds and heuristics}\label{sec:computing}
Theorem~\ref{theorem: existence threshold} gives us the location of the threshold for the appearance of copies of a given poset $(P, \leq_P)$ inside the random poset $(\mathcal{P}(n,p), \subseteq)$ in terms of the parameter 
\begin{align}\label{eq: c_{star} def max min}
c_{\star}(P)=\max_{\boldsymbol{\alpha}\in \triangle_{a(P)}}\quad  \left(\min_{\emptyset\neq Q\subseteq P} \frac{H_{a(Q)}(\partial_Q(\boldsymbol{\alpha}))}{\vert Q\vert }\right).\end{align}
In practice, this parameter is somewhat awkward to compute by hand, even for very small examples. Of course, as it is the maximum of a continuous function over a compact set, we can obtain good computational approximations for its value --- though one should note that the complexity will certainly grow exponentially in $P$ (since we are optimising the value of $\boldsymbol{\alpha}$ over an $a(P)$-dimensional space, and minimising over all $2^{\vert P\vert}-1$ nonempty subsets $Q\subseteq P$).

In this section we prove some general bounds on $c_{\star}(P)$ and discuss heuristics for computing its value exactly --- heuristics that  in particular were used to determine $c_{\star}(P)$ for the examples in the next section. We begin by establishing some useful properties of the set of optimal weightings. Given a poset $P$, let $\mathit{Opt}(P)$ denote the collection of weightings $\boldsymbol{\alpha} \in \triangle_{a(P)}$ for which equality is attained in \eqref{eq: c_{star} def max min}, i.e. such that $c_{\star}(P)=\min\Bigl\{c_{\boldsymbol{\alpha}, P}(Q):\  \emptyset\neq Q\subseteq P \Bigr\}$.
\begin{proposition}\label{prop: opt(p) is convex}
$\mathit{Opt}(P)$ is a convex subset of $\triangle_{a(P)}$.	
\end{proposition}
\begin{proof}
Suppose $\boldsymbol{\alpha}, \boldsymbol{\beta} \in \mathit{Opt}(P)$. Given $\theta\in [0,1]$, consider the weighting $\boldsymbol{\gamma}= \theta \boldsymbol{\alpha}+ (1-\theta)\boldsymbol{\beta}$. Clearly $\boldsymbol{\gamma}\in \triangle_{a(P)}$. Fix any subposet $Q$ with $\emptyset \neq Q\subseteq P$. 
Since $\partial_Q: \ \boldsymbol{x}\mapsto \delta_Q(\boldsymbol{x})$ is a linear operator from $\triangle_{a(P)}$ to $\triangle_{a(Q)}$ and since the map $x\mapsto -x\log x$ is concave, we have 
\begin{align*}
c_{\boldsymbol{\gamma}, P}(Q)&=\frac{1}{\vert Q\vert }H_{a(Q)}(\partial_Q(\boldsymbol{\gamma}))= \frac{1}{\vert Q\vert }H_{a(Q)}\left(\theta \partial_Q\left( \boldsymbol{\alpha}\right) +(1-\theta)\partial_Q\left(\boldsymbol{\beta}\right)\right)\\
&\geq \frac{1}{\vert Q\vert }\left(\theta H_{a(Q)} \left(\partial_{Q}(\boldsymbol{\alpha})\right)   +(1-\theta)H_{a(Q)}\left(\boldsymbol{\partial_Q(\beta})\right)\right)= \theta c_{\boldsymbol{\alpha}, P}(Q)+(1-\theta)c_{\boldsymbol{\beta}, P}(Q)\geq c_{\star}(P).
\end{align*}	
In particular it follows from the definition of $c_{\star}(P)$ in \eqref{eq: c_{star} def max min} that $\boldsymbol{\gamma}\in \mathit{Opt}(P)$. Thus $\mathit{Opt}(P)$ is a convex set as claimed.
\end{proof}

A \emph{poset-automorphism} of $(P, \leq_P)$ is a bijection $\phi: \ P\rightarrow P$ such that both $\phi$ and its inverse are poset homorphisms. Write $\mathrm{Aut}(P)$ for the set of all poset-automorphisms of  $(P, \leq_P)$. Then each $\phi \in \mathrm{Aut}(P)$ induces a permutation on the elements of $\mathcal{A}(P)= \{S_1, S_2, \ldots, S_{a(P)} \}$, with $S_j$ sent to the antichain $\phi(S_j)=\{\phi(i): \ i \in S_j\}$. Similarly, each $\phi\in \mathrm{Aut}(P)$ gives rise to a permutation of the space of weightings $\triangle_{a(P)}$ via permutation of the coordinates, with $\boldsymbol{\alpha}$ sent to the weighting $\phi(\boldsymbol{\alpha})$ defined by $\phi(\alpha)_i=\alpha_j$ if $\phi(S_i)=S_j$.

Given a poset $(P, \leq_P)$, its \emph{reverse} is the poset $(P, \leq_{R(P)})$, where $x\leq_{R(P)}y$ if and only if $y\leq_P x$. Thus the reverse of a poset is simply the poset obtained by reversing all inequalities. We say a poset $(P, \leq_P)$ is \emph{reverse-symmetric} if there exists a bijection $\phi$ from $(P, \leq_P)$ to its reverse $(P, \leq_{R(P)})$ such that both $\phi$ and its inverse are poset homomorphisms. We refer to such a function $\phi$, if it exists, as \emph{reverse automorphism} of $P$, and let $R-\mathrm{Aut}(P)$ denote the set of all reverse automorphisms. Analogously to ordinary automorphisms, reverse automorphisms induce permutations on $\mathcal{A}(P)$ and $\triangle_{a(P)}$.
\begin{proposition}\label{prop: Opt(P) closed under Aut(P) and R-Aut(P)}
\begin{enumerate}	[(i)]
\item $\phi\left(\mathit{Opt}(P)\right)=\mathit{Opt}(P)$ for all $\phi\in \mathrm{Aut}(P)$.
\item If $P$ is reverse-symmetric, then  $\psi\left(\mathit{Opt}(P)\right)=\mathit{Opt}(P)$ for all $\psi\in R-\mathrm{Aut}(P)$.
\end{enumerate}
\end{proposition}
\begin{proof}
	Clearly for any non-empty $Q\subseteq P$  and $\phi\in \mathrm{Aut}(P)$ we have
	\[c_{\phi(\boldsymbol{\alpha}), P}(Q)=c_{\boldsymbol{\alpha}, P} (\phi(Q)).\]
	In particular, $\boldsymbol{\alpha}\in \mathit{Opt}(P)$ if and only if $\phi(\boldsymbol{\alpha})\in\mathit{Opt}(P)$. Similarly, for any $\psi \in R-\mathrm{Aut}(P)$ we have
		\[c_{\psi(\boldsymbol{\alpha}), P}(Q)=c_{\boldsymbol{\alpha}, R(P)} (\psi(Q)).\]
		Now observe that the random poset $R(\P(n,p))$ has the same distribution as $\P(n,p)$. In particular, this immediately implies that for every reverse-symmetric poset $P$,
		 \begin{align*}
		 c_{\star}(P)=c_{\star}(R(P)),
		 \end{align*}
		 and further, since $P$ and $R(P)$ are isomorphic, that $\mathit{Opt}(P)=\mathit{Opt}(R(P))$. Together with the preceding equality, this shows that $\boldsymbol{\alpha}\in \mathit{Opt}(P)$ if and only if $\psi(\boldsymbol{\alpha})\in\mathit{Opt}(P)$.
\end{proof}
\begin{corollary}\label{cor: exists Aut and R-Aut invariant optimal weighting}
	For every poset $P$, there exists an optimal weighting $\boldsymbol{\alpha}\in \mathit{Opt}(P)$ such that $\boldsymbol{\alpha}$ is invariant under the action of automorphisms of $(P, \leq_P)$. Further if $P$ is 
	reverse-symmetric then this optimal weighting can in addition be taken to be invariant under the actions of reverse-automorphisms of $(P, \leq_P)$.
\end{corollary} 
\begin{proof}
Let $\mathrm{AutInv}(P)$ denote the collection of weightings in $\triangle_{a(P)}$ that are invariant under the action of automorphisms $\phi \in \mathrm{Aut}(P)$. Similarly, let $R-\mathrm{AutInv}(P)$ denote the collection of weightings invariant under the action of automorphisms $\phi \in R-\mathrm{Aut}(P)$ (if any such automorphism exists --- otherwise let $R-\mathrm{AutInv}(P)$ denote the whole of $\triangle_{a(P)}$). By considering the uniform weighting $\boldsymbol{u}$, we have that $\mathrm{AutInv}(P)$ and $R-\mathrm{AutInv}(P)$ are non-empty sets, and it is easy to see that both of them form closed convex subsets of $\triangle_{a(P)}$.

Given $\boldsymbol{\alpha}\in \mathit{Opt}(P)$, consider \[\boldsymbol{\tilde{\alpha}}:=\mathbb{E}_{\phi \in \mathrm{Aut}(P)} \phi(\boldsymbol{\alpha}).\]
By Proposition~\ref{prop: Opt(P) closed under Aut(P) and R-Aut(P)}(i) and the convexity of $\mathit{Opt}(P)$, $\boldsymbol{\tilde{\alpha}}$ is an element of $\mathit{Opt}(P)$, and is easily seen to be invariant under the action of automorphisms from $\mathrm{Aut}(P)$.  Thus $\mathrm{AutInv}(P)\cap \mathit{Opt}(P)$ is a non-empty closed and convex subset of $\triangle_{a(P)}$, being the non-empty intersection of closed, convex sets. Repeating the same argument mutatis mutandis starting with $\boldsymbol{\alpha}\in\mathrm{AutInv}(P)\cap \mathit{Opt}(P)$, we obtain in a similar way that $R-\mathrm{AutInv}(P) \cap \mathrm{AutInv}(P)\cap \mathit{Opt}(P)$ is a non-empty (closed and convex) subset of $\triangle_{a(P)}$. The corollary follows.
\end{proof}

For a non-connected poset $P$, we can reduce the computation of $c_{\star}(P)$ to its components.
\begin{lemma}\label{discon}
	If $P$  is a poset which can be partitioned into non-empty sets $P_1$  and $P_2$ of mutually incomparable elements, then $$c_{\star}(P)=\min \{ c_{\star}(P_1),c_{\star}(P_2)\}.$$ 
\end{lemma}
\begin{proof}
	Since the $P_i$ are subposets of $P$  we have  $c_{\star}(P) \leq c':=\min \{c_{\star}(P_1),c_{\star}(P_2)\}$. 
	
	It is easy to see that $c_{\star}(P) \geq c'$ also. Indeed, given any $c<c'$, let $p=e^{-cn}$ and note that $\mathcal P(n,p)=\mathcal P(n,p_1) \cup \mathcal P(n,p_2)$  for some $p_1,p_2$ where $p_1, p_2 \geq e^{-c '' n}$ and $c \leq c'' <c'$.
As $c''<c_{\star}(P_1)$, w.h.p. there is a copy $P'_1$ of $P_1$ in $\mathcal P(n,p_1)$. 
Notice that there is a constant $d=d(P_1)$ so that $\mathcal P(n)$ contains  a copy of $\mathcal P(n-d)$ that avoids $P'_1$.
Further, as $c''< c_{\star}(P_2)$, w.h.p. $\mathcal P(n-d,p_2)$ contains a copy of $P_2$.
	The last two sentences together imply that w.h.p. there is a copy $P'_2$ of $P_2$ in $\mathcal P(n,p_2)$ that does not intersect $P'_1$. Together $P'_1$ and $P'_2$ yield a copy of $P$ in $\mathcal P(n,p)$, as desired.

\end{proof}

Next we consider some special classes of posets $P$ where $\mathit{Opt}(P)$ consists of a single, very simple weighting.
By the remark immediately after the definition of entropy (equation~(\ref{eqdef: entropy})), we have that for all posets $P$, $c_{\boldsymbol{\alpha}, P}(P)\leq \frac{\log a(P)}{\vert P\vert}$
with equality attained if and only if $\boldsymbol{\alpha}$ is the \emph{uniform weighting} $\boldsymbol{u}=\boldsymbol{u}(P)=(\frac{1}{a(P)}, \frac{1}{a(P)},\ldots ,\frac{1}{a(P)})$. In particular 
\begin{align}\label{eq: upper bound on c*(P)}
c_{\star}(P)\leq \frac{\log a(P)}{\vert P\vert} 
\end{align}
holds for all $P$.
\begin{remark}\label{remark: typical copies of P/expectation threshold}
Set $c:=\log (a(P))/\vert P\vert$. Then (by Theorem~\ref{thm:count})  $p=e^{-cn}$ is the threshold at which the expected number of copies of $P$ in $\P(n,p)$ becomes large. Further, for any $P$, the `typical' copies of $P$ in $\P(n)$ are precisely those with weightings close to $\boldsymbol{u}(P)$.
\end{remark}

	\begin{definition}
		A poset $(P, \leq_P)$ is \emph{uniformly balanced} if  $c_{\star}(P)= \frac{\log a(P)}{\vert P\vert}$.
	\end{definition}
\begin{remark}\label{remark: uniformly balancced implies Opt(P)={u}}
By the remark before \eqref{eq: upper bound on c*(P)}, if $P$ is uniformly balanced then necessarily $\mathit{Opt}(P)=\{\boldsymbol{u}(P)\}$.	
A more general form of this result is proved in Proposition~\ref{newprop} below.
\end{remark}

	Checking whether a given poset $(P, \leq_P)$ is uniformly balanced is easy (though possibly tedious if the poset is large): one runs over all possible non-empty subsets $Q$ of $P$, computing the value $c_{\boldsymbol{u}, P}(Q)$ and checking whether or not it is greater or equal to $\frac{\log a(P)}{\vert P\vert}$.

	A finite poset $P$ is \emph{bounded} if it contains a unique $\leq_P$-minimal and a unique $\leq_P$-maximal element. In such posets, we denote the unique minimum and maximum elements by $\min(P)$ and $\max(P)$ respectively. For bounded posets $P$, the upper bound $c_{\star}(P)\leq \log (a(P))/\vert P\vert$ in \eqref{eq: upper bound on c*(P)} may be improved as follows.
	
	\begin{proposition}\label{prop: bound on c_{star}(P) balanced}
	Let $(P, \leq_P)$ be a bounded poset on $\vert P\vert\geq 2$ elements. Let $\mathcal{Q}$ be the collection of subposets $Q$ of $P$ containing both $\min(P)$ and $\max(P)$. Then
	\begin{align}\label{eq: upper bound on c_{star}(P) balanced}
	c_{\star}(P) \leq \max_{x\in [\frac{1}{a(P)}, \frac{1}{2}]} \min\left \{H_2(x), \min_{Q\in \mathcal{Q}}-\frac{1}{\vert Q\vert} \left(2x\log x + (1-2x)\log \left(\frac{1-2x}{a(Q)-2}\right) \right)\right \}.\end{align}
\end{proposition}
\begin{proof}

	 Denote by $Q_-$ and $Q_+$ the one-element subposets of $(P, \leq_P)$ induced by $\{\min(P)\}$ and $\{\max(P)\}$ respectively.  By definition of $c_{\star}(P)$, we have
	\begin{align}\label{eq: upper bound on c_star(P) in terms of Q, Q-, Q+}
	c_{\star}(P) \leq \max_{\boldsymbol{\alpha} \in \Delta_m}\min \left \{ c_{\boldsymbol{\alpha}, P}(Q_-), c_{\boldsymbol{\alpha}, P}(Q_+), \min_{Q\in \mathcal{Q}}c_{\boldsymbol{\alpha}, P}(Q) \right \}.
	\end{align}
	Consider a  weighting $\boldsymbol{\alpha}\in \Delta_m$ of $\mathcal{A}(P)$.  Then by definition
	\begin{align}\label{eq: c-value for Q_- and Q_+}
	c_{\boldsymbol{\alpha}, P}(Q_-)= H_2(\alpha_{\{\min(P)\}}) &&\textrm{ and }&&  c_{\boldsymbol{\alpha}, P}(Q_+)=H_2(\alpha_{\emptyset}). 
	\end{align}
Observe that $\partial_Q(\boldsymbol{\alpha})_{\emptyset}=\alpha_{\emptyset}$ and $\partial_Q(\boldsymbol{\alpha})_{\{\min(P)\}}=\alpha_{\{\min(P)\}}$.\footnote{Note that here and elsewhere, we abuse notation slightly in order to ease the exposition.
Explicitly, if the antichains of $P$ are enumerated as $\mathcal{A}(P)=\{S_1, S_2, \ldots, S_{a(P)}\}$, we use the antichain $S_i$ in place of $i$ as an index for weightings in $\triangle_{a(P)}$. So for example $\alpha_{\{\min(P)\}}:=\alpha_i$, where $S_i$ is the antichain $\{\min(P)\}$.} 
Since $x\mapsto -x\log x$ is strictly concave in $[0,1]$, we have that for any $Q\in \mathcal{Q}$, 
	\begin{align}\label{eq: c-value upper bound  for Q}
	c_{\boldsymbol{\alpha}, P}(Q)\leq \frac{1}{\vert Q\vert }&\Bigl(-\alpha_{\{\min(P)\}}\log (\alpha_{\{\min(P)\}}) -\alpha_{\emptyset}\log (\alpha_{\emptyset} )
	\notag \\&
	- \left(a(Q)-2\right)\left (\frac{1-\alpha_{\{\min(P)\}}-\alpha_{\emptyset}}{a(Q)-2} \right)\log \left(\frac{1-\alpha_{\{\min(P)\}}-\alpha_{\emptyset}}{a(Q)-2 }\right) \Bigr), 
	\end{align}
	with equality holding if and only if $\partial_Q(\boldsymbol{\alpha})_S=(1-\alpha_{\emptyset} -\alpha_{\{\min(P)\}})/(a(Q)-2)$ for all antichains $S\in \mathcal{A}(Q)\setminus\{ \emptyset, \{\min(P)\}\}$. Setting $x:=\alpha_{\emptyset}$, $y:=\alpha_{\{\min(P)\}}$, $R:=\{(x,y)\in [0,1]^2: \ x+y\leq 1 \}$ and combining \eqref{eq: upper bound on c_star(P) in terms of Q, Q-, Q+}, \eqref{eq: c-value for Q_- and Q_+} and \eqref{eq: c-value upper bound  for Q} we get
	\begin{align*}
	c_{\star}(P)\leq \max_{(x,y)\in R} \min \left \{ H_2(x), H_2(y),
	\min_{Q\in \mathcal{Q}}\frac{1}{\vert Q\vert }\left(-y\log y -x\log x- (1-x-y)\log \left(\frac{1-x-y}{a(Q)-2}\right) \right )  \right \}. 
	\end{align*}
	It is then an easy exercise in optimisation to show that the maximum on the right hand side is attained on the line $x=y$, and that for $x=y$ the maximum is obtained for some $x$ satisfying $\frac{1}{a(P)}\leq x \leq \frac{1}{2}$, yielding the claimed upper bound on $c_{\star}(P)$.
\end{proof}
\begin{definition}\label{def: balanced}
	A bounded poset $(P, \leq_P)$ with $\vert P\vert\geq 2$ elements is \emph{balanced} if $c_{\star}(P)=H_2(x_{\star})$, where $x_{\star}=x_{\star}(P)$ is the unique solution in $[1/a(P), 1/2]$ to the equation 
	\begin{align*}
	H_2(x)= -\frac{1}{\vert P\vert } \left(2x\log x + (1-2x)\log \left(\frac{1-2x}{a(P)-2}\right)\right).
	\end{align*}
\end{definition}

\begin{proposition}\label{newprop}
If $P$ is balanced then $\mathit{Opt}(P)=\{\boldsymbol{b}(P)\}$, where $\boldsymbol{b}=\boldsymbol{b}(P)$ is defined so that $b_S:=x_{\star} (P)$ if $S= \emptyset, \{ min (P)\}$, and
$b_S:=(1-2x_{\star} (P) )/(a(P)-2)$ otherwise.
\end{proposition}
\proof


Let $f_1(x)$ and $f_2(x,y)$ be the functions defined by $f_1(x):=-\frac{1}{|P|} \left (2x \log x +(1-2x) \log \left (\frac{1-2x}{a(P)-2 }\right ) \right )$ for $x \in [0,1]$
and $f_2(x,y):=\frac{1}{\vert P\vert }\left(-y\log y -x\log x- (1-x-y)\log \left(\frac{1-x-y}{a(P)-2}\right)  \right)$ for $(x,y)\in R$, where $R:=\{(x,y)\in [0,1]^2: \ x+y \leq 1 \}$. As observed in the proof of Proposition~\ref{prop: bound on c_{star}(P) balanced}, 
\begin{align}\label{eq: unique maximum of min(f, H)}
\max_{(x,y)\in R} \min\left \{ H_2(x), H_2(y), f_2(x,y)\right \} =\max_{x\in [0,1]}\left \{ H_2(x), f_1(x)\right \}=H_2(x_{\star})=f_2(x_{\star}, x_{\star}),
\end{align}
and this common maximum is uniquely attained at $x=y=x_{\star}$.

Consider any $\boldsymbol{\alpha}\in \mathit{Opt}(P)$. Let $x':=\alpha_{\emptyset}$, $y':=\alpha_{\{\min(P)\}}$. Thus $c_{\boldsymbol{\alpha}, P}(\{\max(P)\})= H_2(x')$ and  
$c_{\boldsymbol{\alpha}, P}(\{\min(P)\})= H_2(y')$. Further, by strict concavity of the function $x\mapsto -x\log x$, $c_{\boldsymbol{\alpha}, P}(P)\leq f_2(x',y')$ and this inequality is strict unless $\alpha_S= (1-x'-y')/(a(P)-2)$ for all antichains $S\in \mathcal{A}(P)$ with $S\neq \emptyset, \{\min(P)\}$. Since $\boldsymbol{\alpha}\in \mathit{Opt}(P)$, we thus have

\begin{align*}
c_{\star}(P)=\min_{\emptyset \neq Q\subseteq P} c_{\boldsymbol{\alpha},P}(Q)&\leq \min\left \{ H_2(x'), H_2(y'), f_2 (x',y')\right \}\stackrel{(\ref{eq: unique maximum of min(f, H)})}{\leq} H_2(x_{\star})=c_{\star}(P).
\end{align*}
 
By the uniqueness of the maximum in \eqref{eq: unique maximum of min(f, H)} the last inequality immediately implies $x'=y'=x_{\star}$.  For these values of $x',y'$, the first inequality is strict unless $\alpha_S= (1-2x_{\star})/(a(P)-2)$ for all antichains $S\in \mathcal{A}(P)$, as observed above. Thus $\boldsymbol{\alpha}\in \mathit{Opt}(P)$ implies $\boldsymbol{\alpha}=\boldsymbol{b}$, as desired.
\endproof


The next simple proposition shows that a uniformly balanced poset is balanced; note the reverse is not true (see the next section for an example).  
\begin{proposition}
Suppose $P$ is a bounded poset on $|P|\geq 2$ elements. If $P$ is uniformly balanced then it is balanced.
\end{proposition}
\proof
Let $f_1(x)$ be as in the proof of Proposition~\ref{newprop}. Then by  (\ref{eq: upper bound on c_{star}(P) balanced}) we have
\begin{align}\label{l4}
c_{\star}(P) \leq \max_{x\in [\frac{1}{a(P)}, \frac{1}{2}]} \min \{H_2(x),f_1 (x) \} \leq \frac{\log a(P)}{|P|},
\end{align}
where the rightmost inequality follows from the concavity of the function $x\mapsto -x\log x$. So if $P$ is uniformly balanced there is equality in (\ref{l4}). Thus,
$$c_{\star}(P) =\max_{x\in [\frac{1}{a(P)}, \frac{1}{2}]} \min\Bigl\{H_2(x), f_1(x) \Bigr\} = H_2 (x_{\star})= f_1 (x_{\star}),$$
and so $P$ is balanced.
\endproof

Informally, a poset $P$ is uniformly balanced if the first copy of $P$ to appear in $\P(n,p)$ is a `typical' copy of $P$ in $\P(n)$. On the other hand a poset $P$ is balanced if the first copy $P'$ of $P$ to appear in $\P(n,p)$ is a `squashed' version of a typical copy, in the following sense: let $x_{\star}=x_{\star}(P)$ be as above.
Then the sets $X_{\min(P)}$ and $X_{\max(P)}$ in $P'$ corresponding to the $\leq_P$-extremal elements $\min(P)$ and $\max(P)$ are sitting in layers $x_{\star}n$ and $(1-x_{\star})n$ rather than in layers $\frac{1}{a(P)}n$ and $(1-\frac{1}{a(P)})n$ as they would in a typical copy. (Recall that $\frac{1}{a(P)}\leq x_{\star}(P)\leq 1/2$, so this means $P'$ has been pushed towards the middle layer relative to a typical copy of $P$ in $\P(n)$.)

\smallskip

Giving an explicit value for $c_{\star}(P)$ when $P$ is balanced is not straightforward ---  indeed we have $c_{\star}(P)= H_2(x_{\star})$, where $x_{\star}=x_{\star}(P)$ is  the unique solution to the equation
\[-x\log x -(1-x)\log (1-x) =\frac{1}{\vert P\vert} \left(-2x\log x - (1-2x)\log \left(\frac{1-2x}{a(P)-2}\right) \right)\]
in the interval $[\frac{1}{a(P)},\frac{1}{2}]$. To show $P$ is balanced  likewise entails some non-trivial computations: one must consider the weighting $\boldsymbol{b}= \boldsymbol{b}(P)$ introduced in Proposition~\ref{newprop}. Running over all non-empty subposets $Q\subseteq P$, one must then check that $c_{\boldsymbol{b}, P}(Q)\geq c_{\boldsymbol{b}, P}(P)$, which involves delicate algebraic manipulations of entropic expressions involving $x_{\star}$. This can be done by hand for some small or nicely structured examples, but requires computer assistance for even moderately-sized posets $P$ (unless they are very nicely structured indeed). However if one is content with numerical approximations for $c_{\star}(P)$, then balanced posets are certainly quite easy to handle with the aid of a computer.

\smallskip

What, however, does one do if $P$ is \emph{not} balanced? To prove a lower bound of the form $c\leq c_{\star}(P)$, we must find a `good' $\boldsymbol{\alpha}\in \triangle_{a(P)}$ such that $c_{\boldsymbol{\alpha}, P}(Q)\geq c$ for all non-empty $Q\subseteq P$. This is in principle an $(a(P)-1)$-dimensional problem, but using our earlier observations about the shape $\mathit{Opt}(P)$, we can significantly reduce the dimension of the search-space.

Explicitly, let $\mathcal{A}(P)=\{S_1, S_2, \ldots S_m\}$. Then Corollary~\ref{cor: exists Aut and R-Aut invariant optimal weighting} says there exists $\boldsymbol{\alpha}\in \mathit{Opt}(P)$ such that for every $\phi \in \mathrm{Aut}(P)$ and $S_i\in \mathcal{A}(P)$,
\begin{align}\label{eq: heuristic 1}
\alpha_ {S_i} =\alpha_{\phi(S_i)},
\end{align}
and if $P$ is reverse-symmetric we in addition have that for every $\psi \in R-\mathrm{Aut}(P)$ and $S_i \in \mathcal A (P)$,
\begin{align}\label{eq: heuristic 2}
\alpha_ {S_i} =\alpha_{\psi(S_i)}.
\end{align}
These inequalities can be used to reduce the number of unknown variables when solving the optimisation problem ~\eqref{eq: c_{star} def max min} to determine $c_{\star}(P)$.

Finally, suppose we believe that we can identify a balanced `core' in $P$, that is some unique non-empty $Q_{\star}\subseteq P$ such that $(Q_{\star}, \leq_P)$ is balanced and such that we believe it is the non-existence of a copy of $Q_{\star}$ which is the last obstruction to the appearance of a copy of $P$ ---  i.e. as soon as copies of $Q_{\star}$ exist in $\mathcal{P}(n,p)$, then so will copies of $P$.  This entails
\[c_{\star}(P)=c_{\star}(Q)=c_{\boldsymbol{b}, Q}(Q),\]
where $\boldsymbol{b}=\boldsymbol{b}(Q_{\star})\in \triangle_{a(Q_{\star})}$ is the (unique) optimal balanced weighting of $Q_{\star}$ as in Definition~\ref{def: balanced}.
\begin{corollary}
If	$Q_{\star}$ is a balanced subposet of $P$ with $c_{\star}(Q_{\star})=c_{\star}(P)$, then for every $\boldsymbol{\alpha}\in \mathit{Opt}(P)$,
\begin{align}\label{eq: heuristic 3}
\partial_{Q_{\star}}(\boldsymbol{\alpha})=\boldsymbol{b},
\end{align}
where $\boldsymbol{b}\in \triangle_{a(Q_{\star})}$ is the optimal balanced weighting of $Q_{\star}$. \qed
\end{corollary}
\noindent The heuristic from \eqref{eq: heuristic 3} in conjunction with \eqref{eq: heuristic 1} and \eqref{eq: heuristic 2} can aid our computations of lower bounds for $c_{\star}(P)$ by giving us extra constraints on the coordinates of an optimal weighting $\boldsymbol{\alpha}\in \triangle_{a(P)}$. It is worth noting that the existence of a (presumed) balanced `core' is of course also helpful for obtaining upper bounds on $c_{\star}(P)$, as noted in Proposition~\ref{prop: bound on c_{star}(P) balanced}.

Given our work in this section, and considering the behaviour for small examples, it is natural to wonder (a) whether or not $\mathit{Opt}(P)$ is always invariant under the action of automorphisms of $P$, and (b) whether or not $\mathit{Opt}(P)$ always consists of a single point. The answer to both of these questions turns out to be no.    Let $H_1$ be the poset obtained from $C_7$ by adding two new element $x_1$ and $x_2$, such that $x_1 <_{H_1} y_4$ 
and $y_4 <_{H_1} x_2$, with no other added relations, where $y_4$ is the middle element of the $C_7$.    We next assign all antichains from $C_7$ the same weight $w_1$  and all which contain either $x_1$ or $x_2$ the weight $w_2$. It is now straightforward, but a bit tedious, to check that $c_*(H_1)=c_*(C_7)$ and that this  value is achieved for an  interval of values for $w_2$, thereby giving a negative answer to the first question.

We can answer the second question by using a similar construction.  Let $H_2$ be obtained from $C_{t}$, for odd $t\geq 11$, by adding a copy of $C(1,2)$ and $C(2,1)$, and specifying that the middle element $y$ of the chain is below the copy of $C(1,2)$ and above the $C(2,1)$.  For this poset we can assign a uniform weight to the antichains in the $C_{t}$ and a reversal invariant weight to the antichain which contain elements not in the $C_{t}$. Here the 
weight of the  two maximal elements of the $C(1,2)$ can be made slightly different without changing the threshold.  Apart from working out the explicit entropies this can be seen by considering the up and down-sets  from $y$ in 
$\mathcal{P}(n)$; these are both copies of $\mathcal{P}(n/2)$, and since  our $c$ is strictly less than half of $c_*(C(1,2))$ both will contain copies of  $C(1,2)$.  By Remark \ref{general threshold} we will also have copies where the maximal elements  receive different weights.

\section{Computing $c_{\star}(P)$ in practice: some concrete examples}\label{sec:egs}

\subsection{Chains, stars and wide diamonds}
Let $C(n_1, n_2, \ldots , n_t)$ denote the poset whose elements come from $t$ pairwise disjoint sets $V_1, V_2, \ldots, V_t$ with $\vert V_i\vert=n_i$ and  $x<_P y$ precisely when $x\in V_i$ and $y\in V_j$ for some $i,j$ with $1\leq i < j\leq t$. When $n_1=n_2=\ldots n_t=1$, we write $C_t$ as a shorthand for $C(1,1,\ldots, 1)$. Thus $C_t$ is the chain of length $t$, and is one of the simplest posets there is as far as computing the parameter $c_{\star}(P)$ is concerned.
\begin{theorem}\label{theorem: cstar for chains}
For all integers $t \geq 2$, $C_t$ is uniformly balanced and satisfies\[c_{\star}(C_t)= \frac{\log (t+1)}{t}.\]
\end{theorem}
This result in fact already follows from the work of Kreuter~\cite{Kreuter98}, since every copy of the poset $C_t$ in $\P(n,p)$ is also an embedded copy of $C_t$ viewed as a distributive lattice. Nevertheless, it is worth giving a proof here as an illustration of the general technique for determining $c_{\star}(P)$ when $P$ is uniformly balanced.
 \begin{proof}
	We may identify $C_t$ with the integer set $[t]=\{1,2,  \ldots, t\}$ equipped with the usual order relation $<$. The antichains of $C_t$ are then 
	$\mathcal{A}(C_t)=\{ \{1\}, \{2\}, \ldots, \{t\}, \emptyset\}$. Consider the uniform weighting $\boldsymbol{u}\in \triangle_{a(C_t)}$, and a non-empty subposet $Q\subseteq P$ with elements $i_1<i_2<\ldots <i_{q}$. This 
	subposet is also 	a chain, of length $q$, and its antichains are the singletons from $Q$ together with the empty antichain. Set $i_0:=0$ and $i_{q+1}:=t+1$, and for each $j\in [q +1]$ let $x_j:= (i_j -i_{j-1})/(t+1)$. Then 
	$\boldsymbol{\beta}=\partial_Q(\boldsymbol{u})$ satisfies $\beta_{\{i_j\}}=x_j$ for each $j\in [q]$ and $\beta_{\emptyset}= x_{q +1}$. By elementary properties of entropy, subject to the constraints above, for a fixed $q$ we have that
	$H_{a(Q)}(\boldsymbol{\beta})$ is minimised by $Q=[q]$,  $x_1=x_2=\dots = x_{q}=1/(t+1)$ and $x_{t+1}= 1-q /(t+1)$. Thus, as $q\leq t$, we have
	\begin{align*}
		c_{\boldsymbol{u}, C_t}(Q)
		&= -\frac{1}{q }\sum_{j=1}^{q+1} x_j \log x_j 
		\geq \frac{1}{q}\left(\frac{q}{t+1} \log (t+1) - \left(\frac{t+1-q}{t+1}\right)\log \left(\frac{t+1-q}{t+1}\right)\right)\\
		&=  \frac{1}{t+1}\log ({t+1})- \left(\frac{1}{q}-\frac{1}{t+1}\right)\log \left(\frac{t+1-q}{t+1}\right) \geq \frac{\log (t+1)}{t}=c_{\boldsymbol{u},C_t}(C_t),
	\end{align*}
	where the last inequality follows from the fact the expression on the left hand side is a non-increasing function of $q$ and $q\leq t$.
	It follows from~\eqref{eq: upper bound on c*(P)} that $c_{\star}(C_t)=\log(t+1)/t$ as claimed.
 \end{proof} 
 
Our second general family are the star posets  $C(1,t)$, which have a common threshold  for $t\geq 2$.
\begin{theorem}\label{theorem: cstar for stars}
	Let $x_{\star}$ denote the unique solution in $[\frac{1}{5},\frac{1}{2}]$ to the equation
	\begin{align}\label{eq: def of critical weight for stars}
	(1-x)\log 2 -H_2(x)=0.
	\end{align}	 	
	Then for all $t\geq 2$, we have
	\[c_{\star}(C(1,t))= H_2(x_{\star})\approx 0.5357390.\]	
\end{theorem}
\begin{proof}
	For the upper bound, consider the $2$-leaved star $C(1,2)$. Let $\boldsymbol{\alpha}\in \mathit{Opt}(C(1,2))$, and set $x:=\alpha_{\{\min(C(1,2))\}}$. Then by standard properties of the entropy function, we have
	\begin{align}\label{newlab2}
		c_{\star}(C(1,2))&\leq \min\left \{ c_{\boldsymbol{\alpha}, C(1,2)}(\{\min(C(1,2))\}), c_{\boldsymbol{\alpha}, C(1,2)}(C(1,2)) \right \} \nonumber \\
		&\leq \min\left \{ H_2(x), \frac{1}{3}\left(-x\log x -(1-x)\log\left(\frac{1-x}{4} \right)\right)\right \}.
	\end{align}
	Now the function $F_1: \ x\mapsto H_2(x)$ is increasing in $[0,\frac{1}{2}]$ and decreasing in $[\frac{1}{2}, 1]$, while the function $F_2: \ x\mapsto \frac{1}{3}\left(-x\log x -(1-x)\log\left(\frac{1-x}{4}\right)\right)$ is 
	increasing in $[0, \frac{1}{5}]$ and decreasing in $[\frac{1}{5}, 1]$. Further, we have $F_1(\frac{1}{5})<F_2(\frac{1}{5})$ and 
	$F_1(\frac{1}{2})>F_2(\frac{1}{2})$.  Thus there is a unique solution in the interval $[\frac{1}{5}, \frac{1}{2}]$ to the equation $F_1(x)=F_2(x)$. Rearranging terms, we see that this unique solution is precisely the value $x_{\star}$ 
	from the statement of the theorem. Together with (\ref{newlab2}) this implies
	\begin{align*}
		c_{\star}(C(1,2))\leq  F_1(x_{\star})=H_2(x_{\star}).
	\end{align*}
	Since $c_{\star}(C(1,t))$ is non-increasing in $t$ this establishes the upper bound in the theorem.

	For the lower bound, we use a direct probabilistic argument. Let $c$ be a fixed constant with $c< H_2(x_{\star})$, and let $p=e^{-cn}$. Then by a standard Chernoff bound, w.h.p. $\P(n,p)$ contains an element $A_0$ with 
	$\vert A_0\vert =\lfloor x_{\star} n \rfloor$. Condition on this event, consider then  $Y= \P(n,p)\cap \{X\in \P(n): A_0\subsetneq X\}$. With our conditioning $\vert Y\vert $ is a binomially distributed random variable with expected 
	value
	\begin{align*}
		\left(2^{(1-x_{\star})n}-1\right)p= e^{n \bigl((1-x_{\star})\log 2  - c\bigr)+o(1)}= e^{n\bigl(H_2(x_{\star})- c\bigr)+o(1)}, 
	\end{align*}
	where in the last inequality we use the defining equation \eqref{eq: def of critical weight for stars} for $x_{\star}$. Thus for $c<H_2(x_{\star})$ fixed, the expectation above is (exponentially) large, and a standard Chernoff 
	bound shows that for any fixed $t$, w.h.p. $\vert Y \vert \geq t$. In particular this shows that for any fixed $t$, w.h.p. $\P(n,p)$ contains a copy of $C(1,t)$ (by considering $A_0$ and any $t$ elements from $Y$),  and 
	$c_{\star}(C(1,t))\geq H_2(x_{\star})$ as claimed. 
\end{proof}

The poset $C(1,2,1)$ is in fact the Boolean lattice $\P(2)$, and better known in an extremal setting as the \emph{diamond}.   Extending the methods of the previous proof  we next determine the threshold for the 
appearance for the broader class of \emph{$t$-wide diamonds} $C(1, t, 1)$, $t\geq 2$, which again share a common threshold for $t\geq 2$.
\begin{theorem}\label{theorem: cstar for wide diamonds}
	Let $x_{\star}$ denote the unique solution in $[\frac{1}{6},\frac{1}{4}]$ to the equation
	\begin{align}\label{eq: def of critical weight for wide diamonds}
	2(1-3x)\log 2 -H_2(2x)=0 .
	\end{align}	 	
	Then for all $t\geq 2$, we have
	\[c_{\star}(C(1,t,1))= (1-2x_{\star})\log 2\approx 0.389429.\]	
\end{theorem}
\begin{proof} 
For the upper bound, let us consider the diamond $C(1,2,1)=\P(2)$. This is a bounded poset.  Label the elements of $C(1,2,1)$ as $1,2,3,4$ with $1=\max(\P(2))$ and $4=\min(\P(2))$. 
Applying~Proposition~\ref{prop: bound on c_{star}(P) balanced}, we obtain
\begin{align*}
c_{\star}(\P(2))&\leq \max_{x\in [\frac{1}{6}, \frac{1}{2}]} \min\left\{H_2(x), \min_{Q\in \{ \{1,4\}, \P(2)\}} -\frac{1}{\vert Q\vert} \left(2x\log x + (1-2x)\log \left(\frac{1-2x}{a(Q)-2}\right) \right)\right\}\\
&\leq \max_{x\in [\frac{1}{6}, \frac{1}{2}]} \min \left \{-\frac{1}{2}\left(2x\log x+ (1-2x)\log(1-2x) \right) , -\frac{1}{4}\left(2x\log x+ (1-2x)\log\left(\frac{1-2x}{4}\right) \right\} \right)\\
&=: \max _{x\in [\frac{1}{6}, \frac{1}{2}]} \min \left \{ F_1(x), F_2(x)\right \}.
\end{align*}

Now the function $F_1(x)$ is increasing in $[\frac{1}{6},\frac{1}{3}]$ and decreasing in $[\frac{1}{3}, \frac{1}{2}]$, while the function $F_2(x)$ is decreasing in $[\frac{1}{6}, \frac{1}{2}]$. Further, we have $F_1(\frac{1}{6})<F_2(\frac{1}{6})$ and $F_1(\frac{1}{4})=\frac{3}{4}\log 2>\frac{5}{8}\log 2= F_2(\frac{1}{4})$.  Thus there is a unique solution in the interval $[\frac{1}{6}, \frac{1}{4}]$ to the equation $F_1(x)=F_2(x)$. Rearranging terms, we see that this unique solution is precisely the value $x_{\star}$ from the statement of the theorem.  We then have
\begin{align*}
c_{\star}(\P(2))\leq \max_{x\in [\frac{1}{6},\frac{1}{2}]} \min\left \{F_1(x), F_2(x)\right \}=F_1(x_{\star})=(1-2x_{\star})\log 2,
\end{align*}
where for the equality we used $F_1(x_{\star})=\frac{1}{2}H_2(2x_{\star})+ x_{\star}\log 2$. Since $c_{\star}(C(1,t,1))$ is non-increasing in $t$, this establishes the upper bound in the theorem.

\smallskip

For the lower bound, we use a direct probabilistic argument. Let $c$ be a fixed constant with $c< (1-2x_{\star})\log 2$ and set $p=e^{-cn}$. Let $N$ denote the number of pairs of sets $(A,B)$ with $A,B\in \P(n,p)$, $A\subset B$ and $n-\vert B\vert = \vert A\vert=\lfloor x_{\star} n \rfloor$. We have
\begin{align}
\mathbb{E}N &= \binom{n}{\lfloor x_{\star} n\rfloor, \lfloor x_{\star} n\rfloor}p^2= \exp \left(\left( -2x_{\star}\log x_{\star} - (1-2x_{\star})\log (1-2x_{\star})- 2c \right)n+o(n)\right)\notag \\
&=\exp\left(n\left(2(1-2x_{\star})\log 2 - 2c \right)+o(n)\right),\label{eq: exponent for expected number N of top bottom}
\end{align}
which by our assumption on $c$ tends to infinity as $n\rightarrow \infty$. Now
\begin{align}
\mathbb{E}\left(N^2\right) &\leq  \left(\mathbb{E}N\right)^2 + \left(2\binom{n-\lfloor x_{\star}n\rfloor}{\lfloor x_{\star}n \rfloor} p +1\right) \mathbb{E}N\notag \\
&= \left(\mathbb{E}N\right)^2+ \mathbb{E}N \left(\exp\left(\left( H_2(2x_{\star}) -H_2(x_{\star}) +2x_{\star}\log 2-c\right)n+o(n)\right)+1  \right). \label{eq: exponent for variance of N}
\end{align}
Using $H_2(2x_{\star})= 2(1-3x_{\star})\log 2$, we have
\begin{align}\label{eq: showing variance is small}
\left(2(1-2x_{\star})\log 2 - 2c\right) -\left( H_2(2x_{\star}) -H_2(x_{\star}) +2x_{\star}\log 2-c\right)&= H_2(x_{\star})-c\notag \\
&>H_2(x_{\star})- (1-2x_{\star})\log 2>0.
\end{align}
Combining \eqref{eq: exponent for expected number N of top bottom}, \eqref{eq: exponent for variance of N} and \eqref{eq: showing variance is small} we have that $\mathbb{E} \left(N^2\right)= \left(1+o(1)\right)\left(\mathbb{E}N\right)^2$, and hence that $\mathrm{Var}N=o(\mathbb{E}N^2)$. A simple application of Chebyshev's inequality then tells us that $N$ is concentrated around its mean and in particular that w.h.p. there exists some pair $(A_0, B_0)$ satisfying $A_0,B_0 \in \P(n,p)$, $\vert A_0\vert = n-\vert B_0\vert  \lfloor x_{\star} n\rfloor$ and $A_0 \subseteq B_0$. Conditioning on this event, consider the binomially distributed random variable $Y=\P(n,p)\cap \{X\in \P(n): \ A_0\subsetneq X\subsetneq B\}$. Then with our conditioning $\vert Y\vert$ is a binomially distributed random variable with expected value at least 
\begin{align*}
\left(2^{(1-2x_{\star})n}-2\right)p= e^{n \bigl((1-2x_{\star})\log 2  - c\bigr)+o(1)}. 
\end{align*}
Since $c<(1-2x_{\star})\log 2$,  the expectation above is (exponentially) large, and a standard Chernoff bound shows that for any fixed $t$, w.h.p. $\vert Y\vert \geq t$. In particular this shows that for any fixed $t$, w.h.p. $\P(n,p)$ contains a copy of $C(1,t,1)$ (by considering $A_0, B_0$ and any $t$ elements from $Y$), and $c_{\star}(C(1,t,1))\geq (1-2x_{\star})\log 2$ as claimed. 
\end{proof}

The trick we used in the proofs of the lower bounds for $c_{\star}(P)$ in Theorem~\ref{theorem: cstar for stars} and Theorem~\ref{theorem: cstar for wide diamonds} of first finding suitable images for the top and/or bottom elements of $P$ in $\P(n,p)$ is applicable more generally. Given a poset $Q$, let $C(1,Q)$ denote the poset obtained by adding a new element $b$ to  $Q$ together with the relations $b<q$ for all $q\in Q$. Similarly, let $C(1,Q,1)$ denote the poset obtained from $Q$ by adding two new elements $b,t$ together with the relations $b<q<t$ for all $q\in Q$. 
\begin{proposition}\label{proposition: general lower bounds for existence}
	Let $Q$ be a finite poset.
\begin{enumerate}
	\item[(i)] Let $x_{\star}$ be the unique solution in $[0,1/2]$ to 
	\[ (1-x_{\star})c_{\star}(Q)- H_2(x_{\star})=0.\]
	Then $c_{\star}(C(1,Q))\geq H_2(x_{\star})$.	
	\item[(ii)] Let $x_{\star}$ be the unique solution in $[0,1/2]$ to 
	\[ 2(1-2x_{\star})c_{\star}(Q)-2x\log 2- H_2(x_{\star})=0.\]
	Then $c_{\star}(C(1,Q,1))\geq H_2(x_{\star})$.
\end{enumerate}		
\end{proposition}
Note that in both cases it is easy to see that the solution $x_{\star}$ is unique: we have a decreasing linear function fighting against a concave entropy function that attains its maximum at $x=1/2$, and $c_{\star}(Q)\leq \log 2$, so considering the values of the functions at $0$ and $1/2$ shows the solution will occur in this interval.
\begin{proof}
	Identical to the lower bound proofs in Theorems~\ref{theorem: cstar for stars}, \ref{theorem: cstar for wide diamonds} with the single change that instead of counting the binomially distributed number of points $Y$ in the subcube above $A_0$/between $A_0$ and $B_0$ that we are investigating, we use instead the fact that $p$ is above the threshold for the existence of a copy of $Q$ in a subcube of that dimension.
\end{proof}

Finally, we record bounds on $c_{\star}(P)$ for $P=C_{\ell}(t)= C(t,t, \ldots, t)$, which is the poset obtained for the chain of length $\ell$ by replacing each element by an antichain of size $t$. 
\begin{proposition}\label{proposition: existence threshold for t-blow up of Cell}
	For all $\ell, t\in \mathbb{N}$, we have 
	\[ \frac{\log 2}{\ell}\leq c_{\star}(C_{\ell}(t)) \leq \frac{\log 2}{\ell}+ \frac{\log\left( \ell -(\ell-1)2^{-t}\right)}{\ell t}.\]
\end{proposition}
Note the lower bound is asymptotically tight as $t\rightarrow \infty$.
\begin{proof}
For the lower bound, let $\sqcup_{i=1}^{\ell}A_{i}$ be an equipartition of $[n]$. Suppose $c<(\log 2 )/\ell$ is fixed, and set $p=e^{-cn}$. Let $Y_i= \P(n,p)\cap \{X\in \P(n): \bigcup_{j<i}A_j \subseteq X \subsetneq \bigcup_{j\leq i}A_j \}$. Clearly the $\vert Y_i\vert$ are independent binomially distributed random variables, each with  $\mathbb{E} Y_i =\left(2^{\vert A_i\vert}-1\right)p = \exp\left(\left(\frac{\log 2}{\ell}-c\right)n+o(n)\right)$. In particular our choice of $c$ ensures that w.h.p. $\vert Y_i\vert >t$ for all $i$. 
Taking any $t$ elements from each of the $Y_i$ then yields a copy of $C_{\ell}(t)$ in $\P(n,p)$. It follows that $c_{\star}(C_{\ell}(t))\geq (\log 2)/\ell$ as claimed. 

For the upper bound, we simply appeal to \eqref{eq: upper bound on c*(P)}, noting that $a\left(C_{\ell}(t)\right)= \ell 2^{t}- (\ell-1)$.
\end{proof}

\subsection{Universality}
Our next aim is to establish the existence of a  universality threshold $c_u$, such that for $c$ smaller than $c_u$ almost all fixed posets appear in $\mathcal{P}(n,p)$.

In 1975 Kleitman and Rotschild \cite{KR75} gave a structural description of a typical poset  on $N$ elements.  Given a ground set $V_N$ of $N$ elements, define a class of posets $A_N$ on $V_N$ as follows.   For every member $P$  of $A_N$, we have a partition of 
$V_N$ into three antichains $L_1, L_2, L_3$ such that $\vert \vert L_i\vert -N/4\vert<\sqrt{N}\log(N)$ for $i=1,3$ and $\vert \vert L_i\vert -N/2\vert <\log(N)$ for $i=2$, together with the following poset relations:  for every $x\in L_1$ and $y\in L_3$, $x <_P y$; for every $x\in L_1$ and $y\in L_2$, either $x<_Py$ or  $x$ and $y$ are incomparable in $P$;  likewise for every $x\in L_2$ and $y\in L_3$, either $x<_Py$ or $x$ and $y$ are incomparable in $P$.
\begin{theorem}[Kleitman and Rotschild, 1975]\label{K-R}
	Asymptotically almost every poset on a set $V_N$ of $N$ (labelled) elements belongs to $A_N$, i.e.  $\lim_{N\rightarrow \infty}\frac{\vert \mathcal{P}_N\vert}{\vert A_N\vert} =1$, where $\mathcal{P}_N$ denotes the collection of all posets on $V_N$.
\end{theorem}
One consequence of this theorem is that if we consider the uniform probability measure on $A_N$, by making each relation between elements from $L_1, L_2$  and $L_2,L_3$ exist with probability $1/2$, 
 then the corresponding random poset will be contiguous with the uniform distribution on $\mathcal{P}_N$, in the sense that the asymptotic 0/1 events of the two distributions agree.  See \cite{JLR} for a more detailed discussion of contiguous random models.

Using the  Kleitman--Rotschild theorem we can establish a universality result for the appearance of posets on $N$ elements as subposets of $\P(n,p)$. 
\begin{theorem}\label{universality}
Almost all posets $P$ on $N$ elements satisfy
\[c_{\star}(P) = \frac{\log 2}{3} + O\left(\frac{1}{\log N}\right),\]
where the lower order term is positive.
\end{theorem}
\begin{proof}
By Theorem~\ref{K-R}, it is enough to show that for a uniformly chosen random poset $P$ from $A_N$ we have that w.h.p. $c_{\star}(P)=(\log 2)/3 + O\left(1/\log N\right)$.	
	
For the lower bound, note that every poset $P$ in $A_N$ is clearly a  subposet of $C(N, N,N)=C_3(N)$. Thus $c_{\star}(P)\geq c_{\star}(C_3(N))$, which by Proposition~\ref{proposition: existence threshold for t-blow up of Cell} is at least $(\log 2)/3$.

For the upper bound, it is an easy exercise in discrete probability to show that there is a  constant $b>0$, such that w.h.p. a uniformly chosen random poset $P$ from $A_N$  contains a copy  of   $C(t,t,t)=C_3(t)$, where $t= \lceil b \log N\rceil$.   For such a $P$, we thus have $c_{\star}(P)\leq c_{\star}(C(t,t,t))$, which by Proposition~\ref{proposition: existence threshold for t-blow up of Cell} is at most $(\log 2)/3 + O\left(1/\log N\right)$.	
\end{proof}

\subsection{Small examples}\label{section: small examples}
We now turn to examples of computations of the exact or approximate value of $c_{\star} (P)$ for various small posets $P$, including some of the posets used for our Ramsey results in the next section.

Write $V$ for the $3$-element poset on $\{A,B,C\}$ with  $A<_VB$, $A<_V C$. Thus $V= C(1,2)$. Set also $\Lambda$ to be the reverse of $V$, i.e. the $3$-element poset on $\{A,B,C\}$ with $A>_{\Lambda }B,C$. Theorem~\ref{theorem: cstar for stars} determined $c_{\star}(V)$. Now for any poset $P$, $c_{\star}(P)=c_{\star}(R(P))$ (since $R(\P(n,p))$ and $\P(n,p)$ have the same distribution). Thus Theorem~\ref{theorem: cstar for stars} also determined $c_{\star}(\Lambda)$. Combining this with Theorem~\ref{theorem: cstar for chains} we thus have the existence thresholds for all posets $P$ on at most $3$ elements.

Next, let $\Lambda'$ be the poset obtained from $\Lambda$ by adding two elements $D,E$ and the relations $B>_{\Lambda'}D$ and $C>_{\Lambda'}E$.
Let $Y$ denote the poset on $\{A,B,C,D\}$ defined by the relations $A<_YB$ and  $B<_Y C, D$ (so $Y=C(1,1,2)$.  Let $Y'$ denote the poset obtained from $Y$ by adding two new elements  $E, F$ and the relations $A <_{Y'} E <_{Y'} F$. Let $Y''$ be the poset obtained from $Y'$ by adding four new elements $G, H, I, J$ and the relations $G<_{Y''}A$ and $G<_{Y''}H<_{Y''}I<_{Y''}J$.

Further let $T_2$ denote the binary tree of height $3$, which is the poset obtained from $Y'$ by adding a new element $G$ and the relation $E<_{T_2}  G$. Let $F$ denote the fish-like poset obtained from $T_2$ by identifying the elements $D$ and $F$.

We will also need results about the existence thresholds of the `long $Y$' $C(1,1,1,2)$, of $C(2,3,2)$, of the kite-like poset $C(1,1,2,1)$, of the double diamond $DD$, which is obtained from $C(1,2,2,1)$ by removing one of the inequalities between the elements in the second and third layer. Finally we will need to know the existence thresholds for $C(1,2,1,2,1)$ (a diamond on top of a diamond), and for the poset $H$ defined by the Hasse diagram in Figure~2.

In  Figure \ref{fig:posets} we display  the Hasse diagrams for some of these posets.   Using the bounds from the previous sections we have computed the threshold  for all connected posets on at most four elements, and some of the additional examples which we will use in our results on Ramsey thresholds.  In some cases we have to settle for upper and lower bounds on the threshold.   These results are compiled in Table \ref{tab:th}.
For uniform and balanced posets we state a numerical version of the  exact threshold in the lower bound column of the table.  For general posets we state the best lower bound we have found by a numerical procedure, and the best upper bound found by either using the simple upper bound $\frac{\log a(P)}{\vert P \vert}$  or the best upper bound for a subposet for $P$. In the final column  of the table we state which class the posets belongs to, and we include the label Exact for posets which do not belong to our general classes but for which we can nonetheless determine the exact value of the threshold (in terms of the  solution to an equation involving entropy functions).

We have already seen several examples of families of posets with identical thresholds,  in particular $C(1,t)$  and $C(1,t,1)$  for $t\geq 2$,  and for these we only include the smallest member in the table.  There are a few posets $P$ for  which the bounds on $c_{\star}(P)$ that we obtain are very close, and where the thresholds should in fact be the same, for instance $Y$ and $Y'$. 
\begin{conjecture}
	$c_{\star}(Y)=c_{\star}(Y')$.
\end{conjecture}

\begin{table}[ht]
	\begin{center}
		\begin{tabular}{|r||r|r|r|}
			\hline 
			 Name &	\multicolumn{1}{c|}{L.b.}   & \multicolumn{1}{c|}{U.b.}   & \multicolumn{1}{c|}{Class}   \\ 
			\hline \hline
			$C(2)$	& 0.549306\ldots		&	 	&	Uniform	 \\ \hline 
			$V=C(1,2)$  & 0.53573885\ldots	 	&	 	&	Exact	 \\ \hline 	
			A2  & 0.51986038\ldots	 	&	 	&	Uniform	 \\ \hline 
			$C(2,2)$  & 0.48647753\ldots	 	&	 	&	Uniform	 \\ \hline 											
			$C(3)$  & 0.462098\ldots	 	&	 	&	Uniform	 \\ \hline 	
			 A1  & 0.4620981202 	&   0.4620981203 	&	General	 \\ \hline 	
			 
			 $\Lambda'$  &  0.455914351 &   0.46051702 	&	General	 \\ \hline

			$C(1,2,1)$  & 0.447699551\ldots	 	&	 	&	Balanced	 \\ \hline

			$Y$  & 0.44769950088	 	& 0.44793987	 	&	General	 \\ \hline
			$Y'$  & 0.44769951418	 	&	0.44793987 	&	General	 \\ \hline 							 	
			$T_2$  & 0.4474689916	 	&	0.44793987  	&	General	 \\ \hline 
			
			$F$  & 	0.43238626	&  0.43984289 	&	General	 \\ \hline 			
				
			$C(2,1,2)$  & 0.415888308\ldots	 	&	 	&	Uniform	 \\ \hline 				
			$C(1,2,2)$  & 0.415507009 	 	& 0.4158883 	 	&	General	 \\ \hline 				
			$C(4)$  & 0.402359\ldots	 	&	 	&	Uniform	 \\ \hline 	
			$C(1,1,2,1)$  & 0.3891411 	&  0.38918203 	&	General	 \\ \hline 		
			$C(1,1,1,2)$  & 	 0.3891411	&  0.38918203	 	&	General	 \\ \hline 			
				
			$Y''  $  &  	 0.38890390	&   0.38918203	 	&	General	 \\ \hline 				
			$DD$  & 	0.3816641132\ldots	&  	 	&	Balanced	 \\ \hline 				
			$C(2,3,2)$  & 0.376783	 	& 0.3770081   	 	&	General	 \\ \hline 				
			$\mathcal{P}(3)$  & 0.36356411\ldots	&	 	&	Uniform	 \\ \hline 				
			$C(1,2,1,2,1)$  & 0.3289037390\ldots	 	&	 	&	Uniform	 \\ \hline 				
			$H$  & 0.3250121326	 	& 0.328903	 	&	General	 \\ \hline 				
			
	\end{tabular}\end{center}
	\caption{Thresholds for small posets}
	\label{tab:th}
\end{table}

\begin{figure}[ht]
	\includegraphics[scale=1]{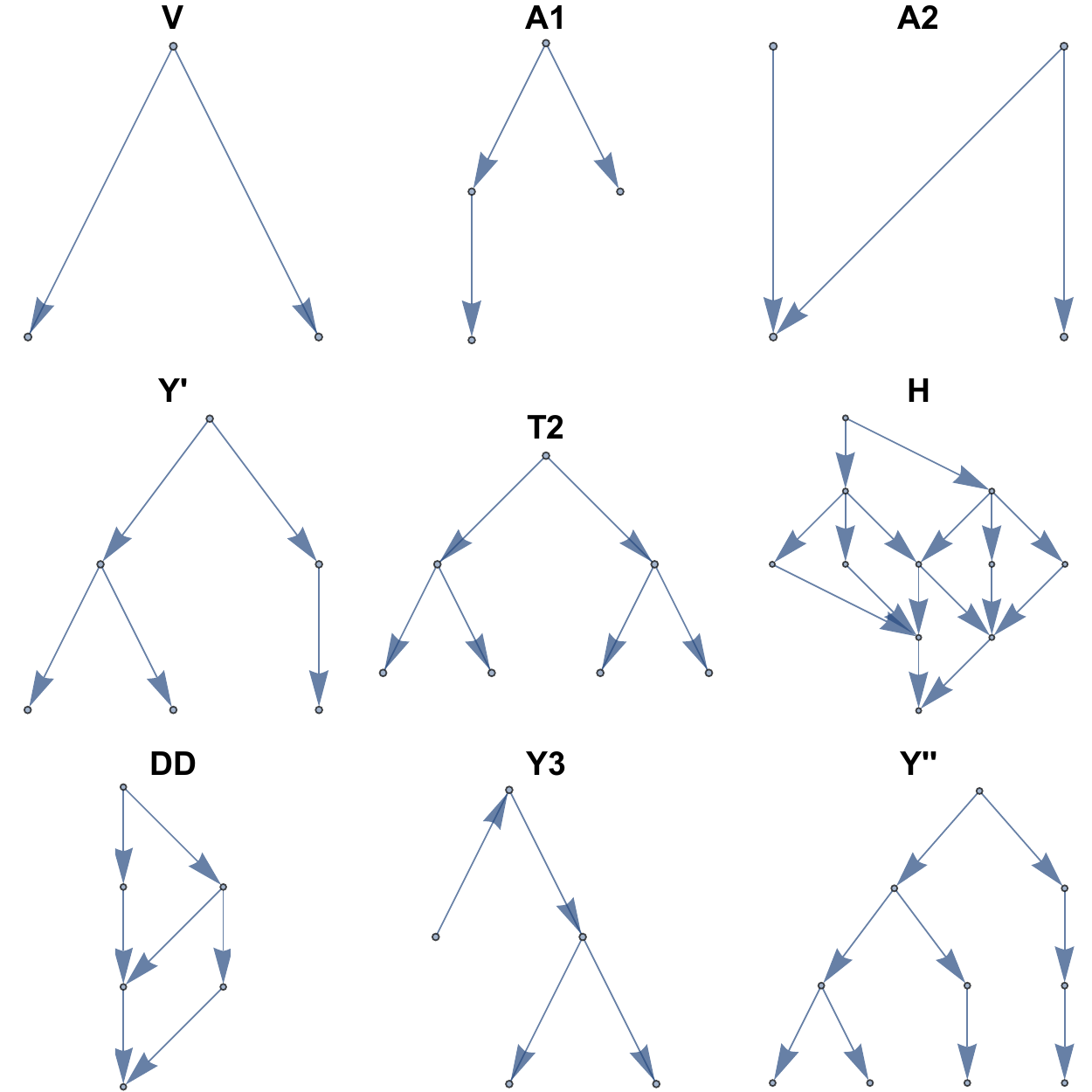}
	\caption{Small posets}\label{fig:posets}
\end{figure}

\section{Ramsey thresholds for posets}\label{sec:ramsey}

\subsection{Ramsey exponents}
Given non-empty posets $P,Q, R$, we say that $R$ is \emph{$(P,Q)$-Ramsey}   if in every $2$-colouring of the elements of $R$, there is either a copy of $P$ in colour $1$ or a copy of $Q$ in colour $2$. We write $R\rightarrow (P,Q)$ if $R$ is  $(P,Q)$-Ramsey, and $R\not\rightarrow (P,Q)$ otherwise. The \emph{poset Ramsey number} $R(P,Q)$ of the pair $(P,Q)$ is defined to be the least $N\in \mathbb{N}$ such that $\P(N)\rightarrow (P,Q)$. Recall from the introduction that this number exists and is finite for every $(P,Q)$. In this section, we consider the problem of determining the range of $p=e^{-cn}$ for which w.h.p.  $\P(n,p)\rightarrow (P,Q)$.

Define the \emph{lower} and \emph{upper Ramsey exponents} of $(P,Q)$ $c_{\mathrm{Ram}^-}(P,Q)$ and $c_{\mathrm{Ram}^+}(P,Q)$ to be 
\[c_{\mathrm{Ram}^-}(P,Q):= \sup\Bigl\{c> 0: \ \P(n,e^{-cn})\rightarrow(P,Q) \textrm{ holds w.h.p.}  \Bigr\}\]
and
\[c_{\mathrm{Ram}^+}(P,Q):= \inf\Bigl\{c>0: \ \P(n,e^{-cn})\not\rightarrow(P,Q) \textrm{ holds w.h.p.}  \Bigr\}.\]
Clearly
\begin{align*}
0\leq c_{\star}\left(\P(R(P,Q)) \right)\leq c_{\mathrm{Ram}^-}(P,Q) \leq c_{\mathrm{Ram}^+}(P,Q)\leq \max \left \{ c_{\star}(P), c_{\star}(Q)\right \},
\end{align*}
so these exponents are well-defined. If $c_{\mathrm{Ram}^-}(P,Q) = c_{\mathrm{Ram}^+}(P,Q)$, then with say that their common value is the \emph{critical Ramsey exponent} for $(P,Q)$, and denote it by $c_{\mathrm{Ram}}(P,Q)$.
\begin{conjecture}\label{conjecture: Ramsey exponents exist}
	For every pair of fixed posets $(P,Q)$, $c_{\mathrm{Ram}}(P,Q)$ exists.
\end{conjecture}
More generally, rather than a pair of posets $(P,Q)$ we may consider Ramsey problems for pairs of families of posets $(\mathcal{P}, \mathcal{Q})$. We write $R\rightarrow (\mathcal{P}, \mathcal{Q})$ if $R$ is  $(\mathcal{P}, \mathcal{Q})$-Ramsey, i.e. if every $2$-colouring of $R$ contains a copy of a member of $\mathcal{P}$ in colour $1$ or a member of $\mathcal{Q}$ in colour $2$. We extend our definitions of poset Ramsey numbers and Ramsey exponents from poset pairs $(P,Q)$ to poset family pairs $(\mathcal{P}, \mathcal{Q})$ in the natural way.

\subsection{General bounds}
Let $P,Q$ be posets such that $P$ has a unique $\leq_P$-maximal element and $Q$ has a unique $\leq_Q$-minimal element. Define the \emph{$Q$--on--$P$} poset $T=T(P,Q)$ by taking disjoint copies of $P$ and $Q$, identifying $\max(P)$ with $\min(Q)$ and adding the relation $p\leq_T q$ for every $p\in P$, $q\in Q$.
\begin{theorem}\label{theorem: upper bound on Ramsey exponent using tower colouring}
For every pair of fixed posets $P,Q$ such that $P$ has a unique $\leq_P$-maximal element and $Q$ has a unique $\leq_Q$-minimal element, we have
\[ c_{\mathrm{Ram}^+}(P,Q)\leq c_{\star}(T(P,Q)).\]	
\end{theorem}
\begin{proof}
Let $c>	c_{\star}(T(P,Q))$. Then for $p=e^{-cn}$, w.h.p. $\P(n,p)$ contains no copy of $T(P,Q)$. Condition on this event and colour the elements of $\P(n,p)$ as follows. Given an element $x\in P(n,p)$, assign it colour $2$ if it is the unique maximal element in a (not necessarily induced) copy of $P$ in $\P(n,p)$, and otherwise assign it colour $1$.

Clearly in this colouring there is no monochromatic copy of $P$ in colour $1$, since by construction the maximal element is in colour $2$. Further, there is no monochromatic copy $Q'$ of $Q$ in colour $2$, otherwise there must be a copy of $P'$ of $P$ in $\P(n,p)$ such that $\max(P')=\min(Q')$. But then $P'\cup Q'$ contains a copy of $T(P,Q)$, a contradiction.
\end{proof}
Given posets $P,Q$ we may define their \emph{lexicographic product} $P\times_{\textrm{lex}}Q$ to be the poset on $P\times Q:=\{(p,q):\ p\in P, q\in Q\}$ with partial order $\leq$ defined by $(p,q)\leq (p', q')$ if and only if either $p<_P p'$ or $p=p'$ and $q\leq_Q q'$.
\begin{theorem}\label{theorem: lower bound on Ramsey exponent using lexicographic product}
	For every pair of fixed posets $P,Q$, we have
	\[ c_{\star}(P\times_{\textrm{lex}} Q) \leq c_{\mathrm{Ram}^-}(P,Q).\]	
\end{theorem}
\begin{proof}
We claim that $P\times_{\textrm{lex}} Q$ is  $(P,Q)$-Ramsey. Indeed consider any $2$-colouring of $P\times Q$. If for any $p\in P$ the set $\{p\}\times Q$ is monochromatic in colour $2$, then this gives us a copy of $Q$ in colour $2$ inside $P\times_{\textrm{lex}} Q$. Otherwise for every $p\in P$ there exists $q_p\in Q$ such that $(p, q_p)$ received colour $1$. Then the set $\{(p,q_p): \ p\in P\}$ gives us a copy of $P$ in colour $1$ inside $P\times_{\textrm{lex}} Q$.

Thus  $P\times_{\textrm{lex}} Q\rightarrow (P,Q)$ as claimed, and the theorem follows immediately from that fact.
\end{proof}

\subsection{Specific posets}
In this subsection, we give bounds on the Ramsey exponents for $(P,Q)$ for various pairs of small posets $P$ and $Q$.

\begin{theorem}[Kreuter~\cite{Kreuter98}]\label{theorem: ramsey exponent for chains}
For all $s,t\geq 2$, $c_{\mathrm{Ram}}(C_s, C_t)=c_{\star}(C_{s+t-1})$.	
\end{theorem}
\begin{proof}
For the upper bound, we have by Theorem~\ref{theorem: upper bound on Ramsey exponent using tower colouring} that $c_{\mathrm{Ram}^+}(C_s, C_t)\leq c_{\star}(T(C_s, C_t))=c_{\star}(C_{s+t-1})$. For the lower bound, observe that by the pigeonhole principle $C_{s+t-1}\rightarrow (C_s, C_t)$.
\end{proof}

\begin{theorem}\label{theorem: ramsey exponent for V}
 $c_{\mathrm{Ram}}(V, V)= c_{\star}(T_2)$.	
\end{theorem}
\begin{proof}

	For the upper bound, we use a slight variant of the colouring given in the proof of Theorem~\ref{theorem: upper bound on Ramsey exponent using tower colouring}. Assign an element $X\in \P(n,p)$ the colour $1$ if there exists $Y,Z\in \P(n,p)$ with $X\subsetneq Y, Z$ (i.e. if $X$ is the minimal element of a copy of $V$ in $\P(n,p)$), and otherwise assign $X$ the colour $2$. By construction, there is no monochromatic copy of $V$ in colour $2$. Suppose now there exists a monochromatic copy of $V$ in colour $1$. By construction of our colouring, this implies the existence of one of the following: a copy of the binary tree $T_2$ of height 3,  a copy of the poset $F$ obtained from $T_2$ by identifying the elements $D$ and $E$, or a copy of $C(1,2,2)$. In particular, this shows
	\begin{align*}
	c_{\mathrm{Ram}^+}(V,V)\leq \max\left \{c_{\star}(T_2), c_{\star}(F), c_{\star}(C(1,2,2)\right \}=c_{\star}(T_2),
	\end{align*}
where the last equality follows from the bounds given in Section~\ref{section: small examples}. For the lower bound, it is easily checked that $T_2\rightarrow (V,V)$, whence $c_{\mathrm{Ram}^-}(V,V)\geq c_{\star}(T_2)$.
\end{proof}
Clearly a poset $H$ is  $(P,Q)$-Ramsey if and only if its reverse $R(H)$ is $(R(P), R(Q))$-Ramsey for the pair of reverse posets $R(P), R(Q)$. Since $R(\P(n,p))$ has exactly the same distribution as $\P(n,p)$, the Ramsey exponents for $(P,Q)$ and $(R(P), R(Q))$ are equal for all pairs $(P,Q)$. In particular, Theorem~\ref{theorem: ramsey exponent for V} also determines  $c_{\mathrm{Ram}}(\Lambda, \Lambda)$. Thus Theorems~\ref{theorem: ramsey exponent for chains}--\ref{theorem: ramsey exponent for V} together determine the critical Ramsey exponents for all pairs $(P,P)$ with $\vert P \vert \leq 3$. For mixed pairs $(P,Q)$, we can give the following bounds on the Ramsey exponents.
\begin{theorem}\label{theorem: other small Ramsey bounds}
	The following hold:
	\begin{enumerate} [(i)]	
		\item  $c_{\star}(Y') \leq  c_{\mathrm{Ram}^-}(C_2,V)\leq  c_{\mathrm{Ram}^+}(C_2,V) \leq c_{\star}(Y)$;  
				
		\item  $ c_{\star}(C(2,3,2))   \leq c_{\mathrm{Ram}^-} (\Lambda, V) \leq c_{\mathrm{Ram}^+} (\Lambda, V) \leq c_{\star}(C(2,1,2))$;
		
		\item$  c_{\star}(Y'')\leq c_{\mathrm{Ram}^-}(C_3,V) \leq  c_{\mathrm{Ram}^+}(C_3,V) \leq c_{\star}(T(C_3, V))=c_{\star}(C(1,1,1,2))$;  
		
		\item  $c_{\star}( C(2,1,2) )  \leq c_{\mathrm{Ram}^-} (\{V,  \Lambda\}, \{V , \Lambda\}) \leq c_{\mathrm{Ram}^+} (\{V,  \Lambda\}, C_2) \leq  c_{\star}(\Lambda')$;
		
		 \item $c_{\star}(DD)  \leq c_{\mathrm{Ram}^-} (\P(2), C_2) \leq c_{\mathrm{Ram}^+} (\P(2), C_2) \leq  c_{\star}(T(C_2, \P(2)))=c_{\star}(C(1,1,2,1))$.
	\end{enumerate}
\end{theorem}

\begin{proof}
	\begin{enumerate}[(i)]
		\item  For the upper bound, Theorem~\ref{theorem: upper bound on Ramsey exponent using tower colouring} implies $c_{\mathrm{Ram}^+}(C_2,V) \leq c_{\star}(T(C_2, V))=c_{\star}(Y)$. For the lower bound, it is 
		easily checked that $Y'\rightarrow (C_2,V)$.
				
		\item For the upper bound, Theorem~\ref{theorem: upper bound on Ramsey exponent using tower colouring} implies $c_{\mathrm{Ram}^+}(\Lambda,V) \leq c_{\star}(T(\Lambda, V))=c_{\star}(C(2,1,2))$. 
		
		For the lower bound, we claim that $C(2,3,2)\rightarrow(\Lambda, V)$. Indeed, suppose the bottom two elements of $C(2,3,2)$ both received colour $1$. Since $C(3,2)\rightarrow (C_1, V)$, this would give us either a 
		$\Lambda$ in colour $1$ or a $V$ in colour $2$. On the other hand, suppose that the bottom two  elements of $C(2,3,2)$ both received colour $2$. Since $C(3,2)\rightarrow(\Lambda, C_1\sqcup C_1)$, this would give   
		us either a $\Lambda$ in colour $1$ or a $V$ in colour $2$.

		We may thus assume that one of the bottom elements of $C(2,3,2)$ receives colour $1$ and the other receives colour $2$. By symmetry, the same is true of the top elements of $C(2,3,2)$. By the pigeonhole principle, 
		at least two elements in the middle layer of $C(2,3,2)$ are in the same colour, say $1$. Thus we have a $C_3$ (and hence a $\Lambda$) in colour $1$. Thus $C(2,3,2)\rightarrow(\Lambda, V)$ as claimed.

		\item For the upper bound, Theorem~\ref{theorem: upper bound on Ramsey exponent using tower colouring} implies $c_{\mathrm{Ram}^+}(C_3,V) \leq c_{\star}(T(C_3, V))=c_{\star}(C(1,1,1,2))$. 
		
		For the lower bound, we claim that $Y''\rightarrow (C_3, V)$. Indeed, suppose the bottom element of $Y''$ is in colour $2$.  If both of the branches of $Y''$ above this minimum element contain elements in colour 
		$2$, then we have a copy of $V$ in colour $2$. Otherwise one of the branches receives only the colour $1$, and hence gives us a copy of $C_3$ in colour $1$. 
		
		Assume therefore that the bottom element of $Y''$ is in colour $1$. One of the branches of $Y''$ above this bottom element is a copy of $Y'$, which as we observed in part (i) is  $(C_2, V)$-Ramsey. Thus in 
		that branch we either get a copy of $V$ in colour $2$ or a copy of $C_2$ in colour $1$, which together with the bottom element of $Y''$ gives us a copy of $C_3$ in colour $2$.
		
		\item  For the upper bound, assign each vertex in $\P(n,p)$ the colour $1$ if it is the top element of a copy of $C_2$ in $\P(n,p)$, and assign it the colour $2$ otherwise. Clearly in such a colouring there can be no 
		copy of  $C_2$ in colour $2$.  Further a copy of $V$ in colour $1$ would require the existence of a copy of $Y$, while a copy of $\Lambda$ in colour $1$ would require the existence of $\Lambda'$ or $\mathcal P(2)$. Thus we 
		have
		\begin{align*}
		c_{\mathrm{Ram}^+} (\{\Lambda, V\}, C_2)\leq \max\left \{ c_{\star}(Y),c_{\star}(\mathcal P(2)), c_{\star}(\Lambda')\right \}=c_{\star}(\Lambda').
		\end{align*}
		
		For the lower, bound, by considering the colour of the middle element, it is easy to see  that $C(2,1,2)$  is $(\{V, \Lambda\}, \{V, \Lambda\})$-Ramsey.

		\item For the upper bound, Theorem~\ref{theorem: upper bound on Ramsey exponent using tower colouring} implies $c_{\mathrm{Ram}^+}(C_2,\P(2 	)) \leq c_{\star}(T(C_2, \P(2)))=c_{\star}(C(1,1,2,1))$.

		For the lower bound, we claim the double diamond $DD$ is $(\P(2), C_2)$-Ramsey. Indeed, suppose the bottom element of $DD$ receives colour $2$. If any element above it is in colour $2$ we have a $C_2$ in colour $2$. 
		Otherwise $DD$ contains a copy of $\P(2)$ in colour $1$. By reverse-symmetry, we are similarly done if the top element of $DD$ receives colour $2$.

		On the other hand, suppose both the bottom and the top elements of $DD$ are in colour $1$. Then if any two of the other elements of $DD$ are in colour $1$ we have a copy of $\P(2)$ in colour $1$. Otherwise, at least three of the `middle' elements of $\P(2)$ are in colour $2$, and two of these will give us a copy of $C_2$ in colour $2$.
		
	\end{enumerate}
	
\end{proof}

Next we turn our attention to the Ramsey problem for the diamond $\P(2)$. Let $H$ be the poset defined by the Hasse diagram in Figure~2.

\begin{figure}[htb!] 
	\begin{center}\footnotesize
		\includegraphics[width=0.9\columnwidth]{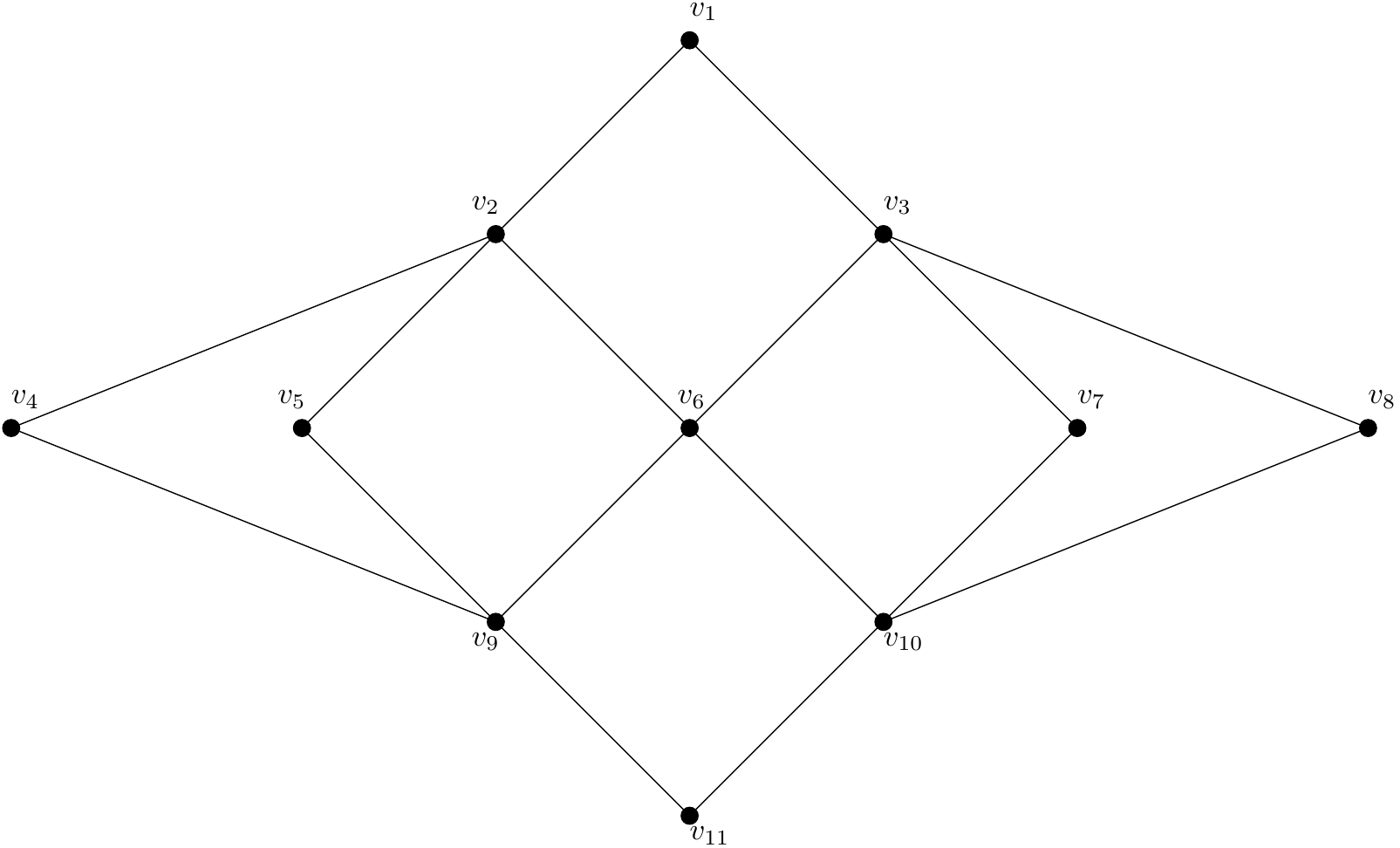}  
		\caption{The Hasse diagram of the poset $H$}
		\label{fig:h1}
	\end{center}
\end{figure}

\begin{theorem}\label{theorem: Ramsey bounds for diamond}
	 $ c_{\star}(H)  \leq c_{\mathrm{Ram}^-} (\P(2), \P(2)  ) \leq c_{\mathrm{Ram}^+} (\P(2), \P(2)  ) \leq c_{\star}(C(1,2,1,2,1))$
\end{theorem}

\begin{proof}
The upper bound on $c_{\mathrm{Ram}^+}(\P(2), \P(2))$ is an immediate consequence of Theorem~\ref{theorem: upper bound on Ramsey exponent using tower colouring} and the fact that $T(\P(2), \P(2))=C(1,2,1,2,1)$. For the lower bound, it suffices to show that $H$ is  $(\P(2),\P(2))$-Ramsey.

Consider any red-blue colouring of the elements of $H$. Suppose for a contradiction that there does not exist a monochromatic copy of $\P(2)$. Without loss of generality we may assume that  \textbf{the maximal element $v_1$ in $H$  is coloured red}. 
\begin{enumerate}[(a)]
\item  If the minimal element $v_{11}$ is coloured red, then at most one  element in $H\setminus\{v_1,v_{11}\}$ is coloured red. Indeed, otherwise we obtain a red copy of $\P(2)$ in $H$. However, then $H\setminus\{v_1,v_{11}\}$  contains two disjoint copies of $\P(2)$, at least one of which is blue, a contradiction. Thus, 
 \textbf{the minimal element $v_{11}$ is coloured blue}.
\item If both $v_9$ and $v_{10}$  are coloured blue, then   $v_2,v_3,v_6$ are coloured red (else we get a blue $\P(2)$). However, then  $v_1,v_2,v_3,v_6$ induce a red $\P(2)$ in $H$. So at least one of $v_9$ and $v_{10}$ is coloured red. 
Without loss of generality assume {\bf{$v_9$ is coloured red}}.

\item Suppose $v_2$ is blue. Then at most one of $v_4,v_5,v_6$ is blue (else we obtain a blue copy of $\P(2)$ with maximal element $v_2$ and minimal element $v_{11}$). However, at most one of $v_4,v_5,v_6$ is red
(else we obtain a red copy of $\P(2)$ with maximal element $v_1$ and minimal element $v_9$). This is a contradiction, so $v_2$ {\bf  is coloured red}.

\item This implies $v_3$ {\bf is coloured blue}
 (else $v_1,v_2,v_3, v_9$ induce a red $\P(2)$).

\item By symmetry with step (c), this implies \textbf{$v_{10}$ is blue}. 

\item Note that if $v_6$ is red, together with $v_1,v_2$ and $v_9$ it induces a red $C_4$, (which contains $\P(2)$ as a subposet). If $v_6$ is blue, together with $v_3,v_{10}$ and $v_{11}$ it induces a blue $C_4$. In either case we obtain a monochromatic copy of $\P(2)$, a contradiction. Thus, $H$ is indeed  $(\P(2),\P(2))$-Ramsey, as claimed.
\end{enumerate}

\end{proof}

Finally we  note that Theorem~\ref{K-R}  shows that most posets have height 3 and this makes it possible to find an interval which contains the Ramsey threshold for almost all posets.
\begin{theorem}
	There exists constants  $\frac{\log2}{5} \leq  c_{ru}^- \leq c_{ru}^+ \leq \frac{\log 2}{3}$ such that almost all posets $P$ on $N$ elements satisfy   
	$c_{ru}^-   \leq c_{\mathrm{Ram}^-} (P, P ) \leq c_{\mathrm{Ram}^+} (P, P  ) \leq  c_{ru}^+ +O(1/\log N)$.
\end{theorem}
\begin{proof}
For every poset $P$ we have that $c_{\mathrm{Ram}^+} (P, P  )\leq c_{\star} (P)$. Thus
Theorem \ref{universality} implies the upper bound.

Recall from 
Theorem~\ref{K-R} that asymptotically almost every poset on a fixed set of $N$ elements belongs to $A_N$, and that every poset $P$ in $A_N$ is a subposet of $C(N,N,N)=C_3(N)$.  So 
	$$
		c_{\mathrm{Ram}^-} (C_3(N)  , C_3(N ))  \leq c_{\mathrm{Ram}^-} (P, P ) \leq  
		c_{\mathrm{Ram}^+} (P, P  ) 		.
	$$
For $c < \frac{\log(2)}{5}$ and $p=e^{-cn}$,  by Proposition~\ref{proposition: existence threshold for t-blow up of Cell},  $\mathcal{P}(n,p)$  w.h.p.
	contains a copy  of $C_5(2N-1)$; any two-colouring of   $C_5(2N-1)$ will contain a monochromatic copy of  $C_3(N)$ and  
	hence also a monochromatic copy of $P$.	 Thus the two exponents lie in the stated interval.
\end{proof}

We  believe that this result can be sharpened to give a single universality exponent.  As we noted in the proof of Theorem \ref{universality}, there is a constant  $b$ such that w.h.p.  a poset $P$ from  $A_{N}$  contains a copy of 
$C_3(t)$ with $t= \lceil b \log N\rceil$.   Hence, in order to be $(P,P)$-Ramsey  the random poset must also be $(C_3(t),C_3(t))$-Ramsey  and hence  $c_{ru}^+ \leq c_{\mathrm{Ram}^+}(C_3(t))$.   Since the sequences $c_{\mathrm{Ram}^-} (C_3(N)  , C_3(N)) $   and $c_{\mathrm{Ram}^+}(C_3(N)),C_3(N))$ are both bounded and non-increasing in $N$, they both converge to limits that give a lower bound on $c_{ru}^-$ and an upper bound on $c_{ru}^+$ respectively. However by Conjecture~\ref{conjecture: Ramsey exponents exist}  these limits should be the same, which would imply the following. 
\begin{conjecture} 
  There exists a constant $c_{ru} = \lim_{N\rightarrow \infty} c_{\mathrm{Ram}^-} (C_3(N))$ such that almost all posets $P$  on $N$ elements have $c_{ru} \leq c_{\mathrm{Ram}^-} (P, P ) \leq c_{\mathrm{Ram}^+} (P, P  ) \leq  c_{ru} +O(1/\log N)$ 
\end{conjecture}

\begin{question}
	What is the value of  $\lim_{N\rightarrow \infty} c_{\mathrm{Ram}^-} (C_3(N))$?
\end{question}

\section{Open problems}\label{sec:open}
In addition to Conjecture~\ref{conjecture: Ramsey exponents exist} about the existence of Ramsey exponents and the obvious problem of tightening our Ramsey results, many other open problems remain.

\begin{question}
{\ }
\begin{itemize}
	\item   Is $C(n,n)$   uniformly balanced for all $n$?
	\item Is $C(n_1,n_2,n_1)$  uniformly balanced if $n_1\geq n_2$?
\end{itemize}
\end{question}

Let $Z_t$  be the poset whose Hasse graph  is obtained from a path on $t$ vertices by giving  the edges alternating directions. The number of antichains  in $Z_t$ is given by the Fibonacci numbers, which implies that $c_{\star}(Z_t)\leq \log\frac{1+\sqrt{5}}{2}+ O(t^{-1})$. On the other hand, $(\log 2)/2$  is a lower bound for $c_{\star}(Z_t)$, since $Z_t$ is a subposet of $C(t/2,t/2)$. This leaves a small gap which it would be nice to close.
\begin{question}
	What is $c_{\star}(Z_t)$?   
\end{question}

Something which we have touched upon in the paper, albeit indirectly, is the size of the connected components of $\mathcal{P}(n,p)$.  
\begin{question}
	What is the size of  the largest connected component of $\mathcal{P}(n,p)$?
\end{question}
At the common threshold for the stars  $C(1,t)$  the components size becomes unbounded. For larger values of $p=\exp(-cn )$ it would be interesting to compare the size of largest component to $2^n e^{-c n}$,  the expected 
number of elements in $\mathcal{P}(n,p)$. 

As we have seen,  once we pass the threshold for the existence of $P$, the collection of `profiles' of copies of $P$ that occur with positive probability in $\P(n,p)$ begins to expand. We have also given examples where this set of embeddings does not consist of a single point, even at $c_{\star}(P)$. It would be interesting to identify conditions which ensure that there is a unique embedding at $c_{\star}(P)$, and to give some quantitative large deviation bounds for the occurring copies of $P$.

Finally, let us note that determining the Ramsey threshold for  $\P(d)$ exactly seems hard, much like  the deterministic question of finding $R(\mathcal{P}(d_1),\mathcal{P}(d_2))$.  In \cite{AxenovichWalzer17} 
various bounds were  given and it was shown that $R(\mathcal{P}(3),\mathcal{P}(3))$ is either 7 or 8.  As part of our own investigation into Ramsey problems for posets we proved the following:
\begin{theorem}
	$R(\mathcal{P}(3),\mathcal{P}(3))=7$ 
\end{theorem}
In order to prove this we created a Boolean satisfiability version of the problem. Here we have one Boolean variable for each element of $\P(d)$. For each $\P(3)$ in $\P(d)$ we create two clauses, expressing that at least one variable in a $\P(3)$ must be set to True  and at least one to False, thereby avoiding a monochromatic copy of $\P(3)$. For $d\leq 6$  satisfying assignments for these Boolean formulae are easily found by a standard SAT-solver  like MiniSat, while for  $d=7$ the formula is found to be unsatisfiable.

\bigskip

\noindent \textbf{Added in proof.}
After  submitting this paper, we learned that Theorem~\ref{thm:count} could be derived from an old result of Stanley. Stanley \cite{Stanley86} showed that for any poset $P$, there is a bijection between the family of antichains in $P$ and $\mathrm{hom}_P(\mathcal{P}(1))$, the family  of homomorphisms from $P$ to $\mathcal{P}(1)$.  As there is a natural bijection between $\mathrm{hom}_P(\mathcal{P}(n))$ and $(\mathrm{hom}_P(\mathcal{P}(1)))^n$, this gives an alternative proof of Theorem~\ref{thm:count}. 

\section*{Acknowledgments}
Much of the research in this paper was conducted whilst the third and fourth authors were visiting Ume{\aa} Universitet. They are grateful to the University for the nice work environment. The authors thank Hao Huang for suggesting an alternative proof of Theorem 1 by using a result of Stanley. The authors are also grateful to the referees for their helpful and careful reviews.


\end{document}